\documentclass[onecolumn,11pt,draft=false]{IEEEtran}
\usepackage{mathtools,color}
\usepackage[hidelinks]{hyperref}
\usepackage{booktabs}
\usepackage{enumitem}
\usepackage{balance}
\usepackage{setspace}
\usepackage{latexsym}
\usepackage{lipsum}
\usepackage{amsfonts}
\usepackage{amssymb,amsmath,bbm}
\usepackage{cleveref}
\usepackage{tikz}
\usetikzlibrary{fit,positioning}
\newtheorem{example}{Example}[section]
\newtheorem{remark}{Remark}
\newtheorem{lemma}{Lemma}
\newtheorem{theorem}{Theorem}
\newtheorem{conjecture}{Conjecture}

\newtheorem{corollary}{Corollary}

\newcommand{\oprocendsymbol}{\hbox{$\bullet$}}
\newcommand{\oprocend}{\relax\ifmmode\else\unskip\hfill\fi\oprocendsymbol}
\allowdisplaybreaks

\begin{document}
\title{Opinion Dynamics with Memory Loss and Communication Delays}

 \author{Somya Singh,
         Sharayu Moharir and
         Neeraja Saharsrabudhe}
  \renewcommand{\thefootnote}{\fnsymbol{footnote}}
  \footnotetext{\\ S. Singh is with Department of Mathematics, Shiv Nadar Institution of Eminence, Delhi-NCR, India, 201314. Email: somya.singh@snu.edu.in\\
  S. Moharir is with Department of Electrical Engineering, Indian Institute of Technology, Bombay, India, 400076.\\
  N. Sahasrabudhe is with Department of Mathematics, Indian Institute of
Science Education and Research, India, 140306.
  }
\maketitle

% \maketitle
% % Optional PDF information
% \ifpdf
% \hypersetup{
%   pdftitle={A Finite Memory interacting P\'{o}lya Contagion Network and Approximating Dynamical Systems},
%   pdfauthor={Somya Singh, Fady Alajaji, Bahman Gharesifard}
% }
% \fi

% The next statement enables references to information in the
% supplement. See the xr-hyperref package for details.

% FundRef data to be entered by SIAM
%<funding-group specific-use="FundRef">
%<award-group>
%<funding-source>
%<named-content content-type="funder-name"> 
%</named-content> 
%<named-content content-type="funder-identifier"> 
%</named-content>
%</funding-source>
%<award-id> </award-id>
%</award-group>
%</funding-group>

% REQUIRED
\begin{abstract}
We propose a novel framework for modeling binary opinions (0 or 1) of individuals connected through a weighted directed network, where edge weights quantify interpersonal influence. Unlike classical models that assume complete access to previously expressed opinions, our framework allows individuals to update their biases using structured memory sets that capture limited and delayed information exchange. To analyze these opinion differences, we introduce a mathematically tractable notion of relative bias between pairs of individuals. The relative biases evolve according to a linear update rule involving past expressed opinions specified by the memory sets. We define the belief of an individual as the probability of expressing opinion 1 and derive a time-delayed dynamical system governing the evolution of network beliefs. We establish its asymptotic behavior and characterize its properties. The framework is further extended to networks containing bots, which maintain fixed biases while influencing neighboring individuals. We quantify the effect of bots by comparing the fixed points of the dynamics in their presence and absence. Finally, simulations illustrate the influence of memory, network structure, and bot interactions on the resulting opinion dynamics.
\end{abstract}

% REQUIRED
\begin{IEEEkeywords}
 Opinion dynamics, network-based interactions, time-delayed dynamical systems, hub-and-spoke model, bot influence.
\end{IEEEkeywords}

\begin{spacing}{1.59}

\section{Introduction}\label{sec:introduction}
Opinion formation in modern networked systems rarely
proceeds under complete information. On online social platforms, an
individual observes only a limited and delayed window of the opinions
expressed by their contacts: content feeds surface a small subset of
recent posts, attention is selective, and different links deliver
different amounts of history. In engineered networks that exchange
binary signals---content-approval votes, congestion indicators,
reputation flags---nodes similarly store only finitely many recent
messages per link, with buffer sizes and delays that vary across links.
At the same time, these networks are increasingly populated by
automated accounts (\emph{bots}) that do not update their opinions but
persistently inject a fixed bias into their neighborhoods, with
documented effects on public-health communication
\cite{JA-EF:18,HL-XM-YZ-MC:23}, political discussion \cite{SH-PR-RM:19}, and
platform ecosystems at large \cite{LN-KC:25}. Understanding how limited,
delayed information exchange shapes collective beliefs---and when a
bot, or a coalition of bots, can measurably steer them---is therefore a
question of both network science and network engineering: it bears
directly on the prediction and control of behavior over information
networks, and on the design of interventions that detect or neutralize
coordinated influence.

Classical opinion dynamics models are not well suited to this regime.
In averaging-based models such as DeGroot \cite{MHD:74} and its
descendants \cite{NEF-ECJ:90,AVP-RT:17,AVP-RT:18}, in voter-type models
\cite{TL:97,TL:99}, and in bounded-confidence models
\cite{RH-UK:02,GD-DN-FA-GW:00}, agents update using the \emph{current}
opinions of their neighbors; Bayesian frameworks
\cite{DA-MD-IL-AO:11,AB-DF:04,KRV:14}, at the other extreme,
typically endow agents with the \emph{entire} history of observations.
Neither captures the intermediate---and practically prevalent---regime
in which each ordered pair of agents shares a finite, link-dependent,
possibly delayed window of past expressed opinions.

In this paper, we propose an opinion dynamics framework built on
\emph{pairwise memory sets}: for each ordered pair $(i,j)$ of agents,
a finitely recent, time-homogeneous set $\mathcal{M}^{(t)}_{ji}$
specifies which past expressed opinions of $j$ are available to $i$
when $i$ updates its disposition toward $j$. This construction
simultaneously models communication delays (limited access to others'
histories) and memory loss (forgetting one's own history), and it
strictly generalizes the finite-memory urn networks of
\cite{SS-FA-BG:22_consensus,SS-FA-BG:22}, which correspond to the special case of
a common memory set. Agents interact over a weighted directed graph
with a row-stochastic interaction matrix, express binary opinions
probabilistically, and update pairwise \emph{relative biases} linearly
from the remembered opinions. The resulting belief dynamics form a
time-delayed discrete-time linear system over the network, whose delay
structure is inherited directly from the memory sets. We then extend
the framework to networks containing bots---agents with fixed bias
(\emph{strength}) $\eta \in (0,1]$ that influence neighbors but never
update---and ask two control-theoretic questions: \emph{when can a bot
shift the asymptotic beliefs of the network toward its own opinion?}
and \emph{when do the influences of multiple competing bots cancel,
leaving the network asymptotically unperturbed?}

\subsection{Contributions}
The main contributions of this paper are as follows.
\begin{enumerate}
    \item \textbf{A pairwise memory-set framework for opinion
    dynamics.} We introduce finitely recent, time-homogeneous memory
    sets indexed by ordered pairs of agents, capturing heterogeneous
    per-link communication delays and memory loss within a single
    tractable model (Section~\ref{sec:model}). The framework admits a
    micro-foundation as an interacting network of finite-memory
    Friedman urns, and recovers existing finite-memory urn networks
    \cite{SS-FA-BG:22_consensus,SS-FA-BG:22} as special cases.
    \item \textbf{Belief dynamics and asymptotics.} We derive the
    induced time-delayed discrete-time dynamical system governing
    network beliefs and establish its asymptotic behavior. For the
    homogeneous case we obtain the fixed point in closed form and show that
    it is independent of the network topology; for the general
    non-homogeneous case we prove existence and uniqueness of the
    fixed point and characterize its dependence on the interaction
    matrix (Section~\ref{sec:model} and \ref{sec:Analysis}).
    \item \textbf{A sharp threshold for single-bot influence.} For a
    network equipped with one bot of strength $\eta_B$, we prove that
    the bot strictly shifts every agent's asymptotic belief toward
    opinion~$1$ if and only if $\eta_B$ exceeds the bot-free homogeneous network
    fixed point $p^{(*)}$, irrespective of the network structure
    (Theorem~\ref{thm:lower_bound_one_bot}). The threshold thus provides a
    topology-independent, quantitative criterion for when a single
    automated account can measurably steer a network.
    \item \textbf{Cancellation conditions for competing bots.} For an
    arbitrary set of bots, we show that their net effect on the
    network vanishes if and only if the bot-influence matrix
    annihilates the vector of centered bot strengths
    (Theorem~\ref{thm:multi_bots}); consequently, non-trivial
    cancellation requires a rank-deficient bot-influence matrix and at
    least one bot on each side of the threshold $p^{(*)}$
    (Corollary~\ref{cor:cancellation}). Beyond its descriptive value,
    this condition suggests a principled recipe for \emph{counter-bot}
    design: neutralizing a detected influence campaign by deploying
    agents whose influence profiles and strengths place the combined
    system in the cancellation set.
    \item \textbf{Bot placement on structured topologies.} For
    hub-and-spoke networks we compute fixed points in closed form
    under two attachment strategies (bot-to-hub versus bot-to-all) and
    show that their relative effectiveness reverses exactly at
    $\eta_B = p^{(*)}$, quantifying how topology mediates the
    \emph{magnitude}---though not the direction---of bot influence
    (Section~\ref{sec:bots}).
    \item \textbf{Numerical validation.} Simulations on general
    directed networks and on hub-and-spoke topologies illustrate the
    convergence of beliefs, the monotone effect of memory size, and
    the predicted bot-induced shifts, including the
    topology-dependence of their magnitude
    (Section~\ref{sec:simulations}).
\end{enumerate}

An implication of the memory-set structure is that the
expressed-opinion process is a finite-order Markov chain, so the
belief dynamics can be analyzed with classical tools while retaining
the delayed linear-system structure familiar from the networked
control literature \cite{EB:06,HH-JLM:26,JP-TL-YL-JC:19}. We emphasize,
however, that the delayed-system machinery is a means to an end: our
focus throughout is on what the analysis reveals about opinion
formation, influence, and its control over networks.

The remainder of the paper is organized as follows.
Section~\ref{sec:related} surveys related work and positions our
contributions. Section~\ref{sec:preliminaries} lists all the notation used throughout the paper.
Section~\ref{sec:model} introduces the memory-set model and derives the
belief dynamics. The asymptotic properties of the belief dynamics are established in Section~\ref{sec:Analysis}.
Section~\ref{sec:bots} extends the model to networks with bots and develops the influence and cancellation results.
Section~\ref{sec:simulations} presents simulations, and
Section~\ref{sec:conclusions} concludes with future directions.

\section{Related Work}
\label{sec:related} 
We organize the discussion along three categories on which our model
differs from prior work: (i) memory and delay structures in opinion
dynamics, (ii) stubborn agents and opinion control, and (iii) bots and
misinformation in online networks. Table~\ref{tab:comparison}
summarizes the comparison with the most closely related models.

\subsection{Memory and Delays in Opinion Dynamics}
Most graph-based non-Bayesian models update opinions from the current
state of the network. In the DeGroot model \cite{MHD:74} and the
Friedkin--Johnsen model \cite{NEF-ECJ:90}, each agent averages its
neighbors' current opinions (in the latter, anchored to its initial
opinion); voter-type models \cite{TL:97,TL:99} copy a
neighbor's current opinion; bounded-confidence models
\cite{RH-UK:02,GD-DN-FA-GW:00,WQ-GC-AS:14,LT-YW-MP-JL:26} restrict
interaction to agents whose current opinions are sufficiently close.
Bayesian models \cite{DA-MD-IL-AO:11,AB-DF:04,KRV:14} instead
condition on entire observation histories. Memory structures
interpolating between these extremes have been studied through
finite-memory P\'olya urns \cite{FA-TF:94}, and through interacting
networks of such urns for consensus \cite{SS-FA-BG:22_consensus} and contagion
\cite{SS-FA-BG:22}; in these works, however, all agents share a common
memory set. Our framework generalizes this structure by indexing memory
sets by \emph{ordered pairs} of agents, so that different links carry
different, possibly delayed, windows of past opinions---the natural
abstraction of heterogeneous buffers, link qualities, and selective
attention. The induced belief dynamics form a time-delayed
discrete-time linear system, connecting our analysis to stability
theory for delayed systems \cite{EB:06,HH-JLM:26,JP-TL-YL-JC:19}; in
contrast to that literature, the delay pattern here is not exogenous
but generated by the memory sets, and our interest is in the
network-level fixed points and their perturbation by bots rather than
in stabilization per se.

\subsection{Stubborn Agents and Opinion Control}
A substantial literature studies agents that resist updating. Ghaderi
and Srikant \cite{JG-RS:12} endow each agent with a stubbornness
parameter $\alpha_i \in [0,1]$ weighting its initial opinion; fully
stubborn agents ($\alpha_i = 1$) also appear in voter-type models
\cite{EY-AO-DA-AS-AS:13} and in the ``radical groups'' and ``charismatic
leaders'' of Hegselmann and Krause \cite{RH-UK:02}. A related
thread treats opinions as a resource to be optimized, from influence
maximization \cite{DK-JK-ET:03} to opinion maximization and
susceptibility-based intervention
\cite{NP-DB-MC:22,AG-ET-PT:13,RA-JK-DP-CT:18}. Our bots are categorically
distinct from ordinary agents rather than endpoints of a stubbornness
continuum, carry a strength $\eta_k \in (0,1]$ toward a binary
opinion, and act through the interaction matrix under per-link memory.
Two results, to our knowledge, have no direct analogue in this
literature: the sharp, topology-independent threshold
$\eta_B > p^{(*)}$ for a single bot to shift the network
(Theorem~\ref{thm:lower_bound_one_bot}), and the exact algebraic cancellation
condition for multiple competing bots
(Theorem~\ref{thm:multi_bots}), which yields both an impossibility
result (full-rank bot-influence matrices admit only trivial
cancellation) and a constructive criterion for counter-bot placement.

\subsection{Bots and Misinformation in Online Networks}
Empirical and simulation studies document the influence of automated
accounts on public health communication \cite{JA-EF:18,HL-XM-YZ-MC:23},
political discussion \cite{SH-PR-RM:19}, and platform behavior at scale
\cite{LN-KC:25}, as well as indirect bot influence mediated by recommender
systems \cite{NP-DB-MC:22} and the spread of (mis)information over
social networks \cite{DA-AO-AP:10}. These studies establish the
phenomenon but generally do not yield closed-form conditions relating
bot parameters to their asymptotic effect. Our model complements them
with an analytically tractable framework in which the effect of bot
strength, bot placement, network topology, and memory size on
asymptotic beliefs can be computed exactly (Sections~\ref{sec:bots}
and~\ref{sec:simulations}), providing a theoretical benchmark against
which empirical bot-influence findings can be interpreted.

% ---------------------------------------------------------------------
\begin{table*}[!t]
\caption{Comparison with the most closely related opinion dynamics
models. FJ: Friedkin--Johnsen. BC: bounded confidence.}
\label{tab:comparison}
\centering
\scriptsize
\setlength{\tabcolsep}{4pt}
\begin{tabular}{p{2.1cm} p{1.2cm} p{2.5cm} p{2.5cm} p{2.5cm} p{1.9cm} p{3.4cm}}
\toprule
Model & Opinion & Information used & Memory/delay & Persistent agents &
Graph & Closed-form influence conditions \\
\midrule
DeGroot \cite{MHD:74} & continuous & current opinions & none & no &
directed, weighted & --- \\
FJ \cite{NEF-ECJ:90} & continuous & current + initial & none &
partially stubborn & directed, weighted & fixed point vs.\
stubbornness \\
BC \cite{RH-UK:02,RH-UK:15} & continuous & current opinions
(proximal) & none & radicals / leaders \cite{RH-UK:15} & implicit
(confidence sets) & --- \\
Stubborn agents \cite{JG-RS:12} & continuous & current opinions &
none & continuum $\alpha_i \in [0,1]$ & undirected & equilibrium
characterization \\
Stubborn voter \cite{EY-AO-DA-AS-AS:13} & binary & current opinion of sampled
neighbor & none & fully stubborn & directed & stationary averages \\
Finite-memory urn networks \cite{SS-FA-BG:22_consensus,SS-FA-BG:22} & binary &
common finite window & common memory set & no & directed, weighted &
fixed point (homogeneous) \\
\textbf{Our Model} & \textbf{binary} & \textbf{per-link finite
windows} & \textbf{pairwise memory sets (delays $+$ loss)} &
\textbf{bots with strength $\eta_k$} & \textbf{directed, weighted} &
\textbf{threshold $\eta_B>p^{(*)}$; multi-bot cancellation $\widetilde{S}_B \eta_{p^{(*)}} = \mathbf{0}_N$} \\
\bottomrule
\end{tabular}
\end{table*}
\section{Preliminaries and Notation}\label{sec:preliminaries}
We denote the transpose, spectral radius of a matrix $Q$ by $Q^{\mathbf{T}}$ and $\rho(Q)$ respectively. Notation $[Q]_{ij}$ denotes the entry in the $i$th row and $j$th column of $Q$. The $N\times N$ identity and zero matrix are denoted by $I_{N\times N}$ and $0_{N \times N}$ respectively. Also, we often write a matrix as $[\quad]_{N \times M}$, where $N$ and $M$ are the number of rows and columns respectively. For a vector $\mathbf{x}$, its $i$-th component is denoted by $\mathbf{x}_{i}$. Given $\mathbf{x}, \mathbf{y}\in\mathbb{R}^{n}$, we write $\mathbf{x}\preceq \mathbf{y}$ iff $\mathbf{x}_i\leq \mathbf{y}_i$ for all $i \in \{1,\cdots,N\}$, $\mathbf{x}\prec \mathbf{y}$ iff $\mathbf{x}_i< \mathbf{y}_i$ for all $i \in \{1,\cdots,N\}$. The relations $\mathbf{x}\succeq \mathbf{y}$ and $\mathbf{x}\succ \mathbf{y}$ are defined analogously. The $N\times 1$ column vectors whose entries are all equal to one and zero are denoted by $\mathbf{1}_{N}$ and $\mathbf{0}_{N}$, respectively. Given two matrices $Q,R \in \mathbb{R}^{n \times m}$ we denote $Q \circ R$ as the matrix given by entry-wise multiplication of $Q$ and $R$ (i.e., $[Q \circ R]_{ij} = [Q]_{ij}[R]_{ij}$). Also, $Q<R$ iff $[Q]_{ij}<[R]_{ij}$ for all entries $i \in \{1,\cdots,n\}$ and $j \in \{1,\cdots,m\}$. We similarly define the relations $Q\leq R$, $Q\geq R$ and $Q>R$. Unless stated otherwise, all networks considered in this paper are connected and  consist of $N$ nodes. 

\section{The Model}\label{sec:model}
We consider a population of $N$ individuals placed on the vertices of a directed weighted network $\mathcal{G}_{N}$. At time $t=0$, each individual is assigned an equal \emph{initial bias}  $X_{0} \in (0,1]$. We also define $X_{ij,t}\in [0,1]$ to be the \emph{relative bias} of individual $i$ with respect to individual $j$. By convention, we assume that $X_{ij,0}= X_{0}$ for all $i,j \in \{1,\cdots, N\}$. At each subsequent time step $t > 0$, all individuals simultaneously express a binary opinion $Z_{i,t} \in \{0, 1\}$, which depends on the relative bias of the individual $i$ with respect to other individuals in the network at time $t-1$. Based on these expressed opinions, each individual $i$ updates all its relative biases. The association between $Z_{i,t}$ and $X_{ij,t-1}$'s can be mathematically written as 
\begin{align}\label{eqn:draw_general}
Z_{i,t} = \begin{cases}
    1 & \textrm{w.p.} \quad f_{i}(X_{i1,t-1},\cdots,X_{iN,t-1})\\
    0 & \textrm{w.p.} \quad 1-f_{i}(X_{i1,t-1},\cdots,X_{iN,t-1}),
\end{cases}   
\end{align}
where $f_{i}:[0,1]^{N} \to (0,1)$ is a function that governs the dependence of $X_{ij,t}$'s on the inherent opinions of other individuals in the network. The expression for $f_{i}$ captures the influence of various individuals in the network on the opinion update (both bias and expressed) of the $i$th individual.

Additionally, we incorporate communication delays and/or memory loss in this network. The former refers to an individual having access to a limited history of expressed opinions from other individuals, while the latter represents an individual itself not remembering its entire history of expressed opinions. To this end, we give the formal definition of the limited information available to agents of the network (we refer to this as memory).

In this paper, we restrict our attention to ``finitely recent'' and ``time-homogeneous'' memory sets. We denote by $\mathcal{M}_{ji}^{(t,T)}$ the memory set consisting of all the time instances less than $t$ for which the expressed opinion of agent $j$ is communicated to agent $i$ when updating the relative bias of $i$ with respect to $j$ at time $t$. Here $T$ denotes a fixed time instant at which information loss starts in the model. We denote $T$ to be the ``memory activation time" of the model. By finitely recent memory $\mathcal{M}^{(t,T)}_{ji}$, we mean that there exists a finite $K_{ji}>0$ such that $\max_{n \in \mathcal{M}^{(t,T)}_{ji}}(t-n)<K_{ji}$. We say that a collection of memory sets $\{\mathcal{M}_{ji}^{(t,T)}\}_{j,i\in \{1,\cdots,N\},t\geq 0}$ is time-homogeneous when:

 \begin{itemize}
 \item[(1)] For time $t=0$, all the memory sets are empty.
 \item[(2)] For all $t\in\{1,\cdots,T\}$, the memory sets $\mathcal{M}_{ji}^{(t,T)}=\{1,\cdots,t-1\}$ for all $i,j \in \{1,\cdots,N\}$.
 
     \item[(3)] For $t> T$, the cardinality of the memory sets $|\mathcal{M}_{ji}^{(t,T)}| =M \leq T$ for all $i,j \in \{1,\cdots,N\}$. Here, $M$ is referred to as the ``memory'' of the network. 
     
     \item[(4)] For each time instant $t> T$, we place the elements of $\mathcal{M}_{ji}^{(t,T)}$ in decreasing order and denote the $k$th  element of the set $\mathcal{M}_{ji}^{(t,T)}$ by $\mathcal{M}_{ji}^{(t,T)}(k)$. Then, for time-homogeneous memory sets, we have $\mathcal{M}_{ji}^{(t+h,T)}(k) - \mathcal{M}_{ji}^{(t,T)}(k) = h$ for all $h>0$ and $k \in \{1,\cdots,M\}$. 
 \end{itemize}
 The simplest non-trivial example for a finitely recent time-homogeneous memory set is $\mathcal{M}^{(t,M)}=\{t-1,\cdots,t-M\}$ for all $t>M$. We refer the reader to \cite{FA-TF:94} where this memory set is used to study finite memory P\'{o}lya urns. For ease of notation, hereafter we write the elements of finitely recent time-homogeneous memory sets $\mathcal{M}_{ji}^{(t,T)}$ as $\{t-k_{ji}^{(1)},\cdots, t-k_{ji}^{(M)}\}$, where $M$ is the cardinality of each memory set, and $1\leq k_{j,i}^{(1)}<k_{j,i}^{(2)}<\cdots<k_{j,i}^{(M)}\leq T$ for all $i,j \in \{1,\cdots,N\}$. Furthermore, we keep $T$ fixed throughout this paper and therefore omit its usage in the superscript of $\mathcal{M}_{ji}^{(t,T)}$ hereafter.     
% Similarly, we say that the memory $\mathcal{M}^{(t)}_{i,j}$ is finitely distinct iff $\max_{n \in \mathcal{M}^{(t)}_{i,j}}n<K$ for some finite $K>0$.
% refer to SIAM paper here

 Next, we consider the following linear form for \eqref{eqn:draw_general} to obtain the expressed opinions of the agents using relative biases from the previous time step:
\begin{align}\label{eqn:draw_S}
Z_{i,t} := \begin{cases}
1 & \sum_{j=1}^{N}s_{ij}X_{ij,t-1}\\\\
0 & 1-\sum_{j=1}^{N}s_{ij}X_{ij,t-1},
\end{cases}   
\end{align}
where $s_{ij}$ is the $(i,j)$th entry of an $N \times N$ row-stochastic matrix denoted by $S_{N}$, and $X_{ij,t}$ is the inherent opinion of individual $j$ known to individual $i$. From a graph-theoretic point of view, this model can be visualized as individuals placed on the vertices of a directed network $\mathcal{G}_{N}$ equipped with a weighted adjacency matrix $S_{N}$. The edge directed from individual $j$ to individual $i$ has a weight $s_{ij}$ and represents the ``influence'' of individual $j$ on the opinions formed by individual $i$. We consider the following linear form for the dependence of $X_{ij,t}$ on the expressed opinions and memory sets:
\begin{align}\label{eqn:ratio_X}
X_{ij,t} := c_{1,M} + c_{2,M}\hspace{-0.2cm}\sum_{n \in \mathcal{M}^{(t)}_{ji}}\hspace{-0.2cm}Z_{j,n}, 
\end{align}
where $|\mathcal{M}_{ji}^{(t)}|=M$. Since the relative bias $X_{ij,t}\in [0,1]$, the constants $c_{1,M}$ and $c_{2,M}$ in \eqref{eqn:ratio_X} must obey the following inequalities at all time instants:
\begin{align}\label{eqn:c_bounds}
\max (0,-c_{2,M}K)\leq c_{1,M} \leq \min (1,1-c_{2,M}K) \quad \textrm{for all} \quad K \in \{0,1,\cdots,M\}.  
\end{align}
This simplifies as follows:
\begin{itemize}
    \item Case I: $c_{2,M} \geq 0$. Then, $\max (0,-c_{2,M}K)=0$, and $1-c_{2,M}K$ decreases in $K$ so minimum occurs at $K=M$. So, \eqref{eqn:c_bounds}
    reduces to 
    \[ 0 \leq c_{1,M} \leq 1-c_{2,M}M \]
    \item Case II: $c_{2,M} < 0$. In this case, by an argument similar to the above, \eqref{eqn:c_bounds}
 reduces to
 \[ -c_{2,M}M\leq c_{1,M} \leq 1 \]
\end{itemize}
Thus, the model's feasibility depends on $|c_{2,M}M|$. Furthermore, we assume that $ c_{2,M} \in (-1/M, 1/M) - \{0\}$, for it to represent the scaling factor for  the sum of expressed opinions $Z_{j,n}$ across the $M$ memory time stamps in the memory set $\mathcal{M}_{ji}^{(t)}$. Note that, if we set the initial bias $c_{1,M}=0$, then with a positive probability, we can have $X_{ij,1} =0$ for all $i,j \in \{1,\cdots,N\}$. This sample path would lead to $Z_{i,t} = 0$ for all $i \in \{1,\cdots, N\}$ and $t\geq 1$, which would give $X_{ij,t} =0$ for all $i,j \in \{1,\cdots,N\}$ and $t\geq 1$. To ensure that our model never enters such absorbing states, a constant positive  ``drift'' of $c_{1,M}$ must be present in \eqref{eqn:ratio_X}. The removal of this drift factor $c_{1,M}$ in \eqref{eqn:ratio_X} significantly alters the long-run behavior of $X_{ij,t}$'s. One such consensus model is discussed in \cite{SS-FA-BG:22_consensus}, where the underlying Markov chain has two absorbing states and the limiting distribution depends on the initial conditions of the model. 

A simplified version for \eqref{eqn:ratio_X} is when all the memory sets are same i.e, $\mathcal{M}_{ij}^{(t)} = \mathcal{M}^{(t)}$ for all $i,j \in \{1,\cdots,N\}$. In this case,
\begin{align}\label{eqn:ratio_simplified}
X_{ij,t} = c_{1,M} + c_{2,M}\sum_{n \in \mathcal{M}^{(t)}}Z_{j,n}.   
\end{align}
Note that, here the relative bias of individual $i$ with respect to $j$ is same for all $i \in \{1,\cdots, N\}$. Therefore, we can define bias of  individual $j$ to be $X_{j,t}:= X_{ij,t}$ for all $i \in \{1,\cdots,N\}$. An analogous special case of \eqref{eqn:ratio_simplified} has already been studied in \cite{SS-FA-BG:22} for an interacting network of P\'{o}lya urns with memory $\mathcal{M}_{ij}^{(t)} = \{t-1,\cdots,t-M\}$ for all urns $i,j$ in the network. The bias of an individual $X_{i,t}$ for this setup is defined as the ratio of red balls in urn $i$ at time $t$.

We now develop the framework for analysis of \eqref{eqn:ratio_X}. The main objective here is to obtain a \emph{time-delayed discrete time recursion} for a suitably defined parameter. To this end, for an individual $i$ in the network, we define its \emph{belief} at time $t>T$ as $P_{i}(t):= P(Z_{i,t}=1)$ and it can be computed using \eqref{eqn:draw_S} and \eqref{eqn:ratio_X} as follows:
\begin{align}\label{eqn:P_calc}
P_{i}(t) &= E [E[Z_{i,t}|\mathcal{F}_{t-1}]] = E[\sum_{j=1}^{N}s_{ij}X_{ij,t-1}] = c_{1,M} + c_{2,M}\sum_{j=1}^{N}\sum_{n \in \mathcal{M}^{(t-1)}_{ji}}\hspace{-0.3cm}s_{ij}P_{j}(n),
\end{align}
where $\mathcal{F}_{t-1}$ is the $\sigma$-algebra generated by the random variables $\{ Z_{1,n}, Z_{2,n}, \cdots, Z_{N,n} : n \in \bigcup\limits_{j,k=1}^{N}\hspace{-0.02cm}\mathcal{M}_{jk}^{(t)} \}$.  It is now easy to see that the belief of individual $i$ at time $t$ (given by $P_{i}(t)$) in \eqref{eqn:P_calc} depends on past beliefs held by itself as well as the other individuals in the network through the interaction parameters $s_{ij}$'s. The memory sets $\mathcal{M}_{ij}^{(t)}$'s determine the past beliefs of individual $j$ which influence the belief of individual $i$ at time $t$. Furthermore, due to these memory sets being finitely recent and time-homogeneous, we can write $\mathcal{M}_{ji}^{(t)}= \{t-k_{ji}^{(1)},\cdots,t-k_{ji}^{(M)}\}$ to obtain an alternate form of \eqref{eqn:P_calc}:
\begin{align}\label{eqn:P_eqn}
P_{i}(t) = c_{1,M} + c_{2,M}\sum_{j=1}^{N}s_{ij}\hspace{-0.5cm}\sum_{n \in \{k_{ji}^{(1)},\cdots, k_{ji}^{(M)}\}}\hspace{-0.7cm}P_{j}(t-1-n).    
\end{align}
We now obtain a matrix form for \eqref{eqn:P_calc}, by defining $\widetilde{P}_{t}^{(L)}:=(P(t),\cdots,P(t-L+1))^{\mathbf{T}}$ for all $t>T$, where $L := \max\limits_{n}\{n \hspace{0.1cm}|\hspace{0.1cm} t-n \in \bigcup_{i,j=1}^{N}\mathcal{M}_{ji}^{(t)} \hspace{0.2cm} \textrm{for all} \hspace{0.2cm} t>T\}$ and $P(t):= (P_{1}(t),\cdots,P_{N}(t))^{\mathbf{T}}$. The time instant $L$ refers to the ``memory depth'' of the model and is the maximum possible time-delay present in \eqref{eqn:P_eqn}. Also, it follows from the definition of memory depth that $L<T$. We now write the following recursion equation in the column vector $\widetilde{P}_{t}^{(L)}$ using \eqref{eqn:P_calc}:
\begin{align}\label{eqn:X}
\widetilde{P}_{t}^{(L)} = J_{N}^{(L)}\widetilde{P}_{t-1}^{(L)} + c_{1,M}e_{N,N L}
\end{align}
where
\begin{align}\label{eqn:J_matrix}
 J_{N}^{(L)}= \left[ 
\begin{array}{c|c|c|c} 
  c_{2,M}S_{N}\circ \mathbf{M}_{N}^{(1)}& c_{2,M}S_{N}\circ \mathbf{M}_{N}^{(2)}& \cdots & c_{2,M}S_{N}\circ \mathbf{M}_{N}^{(L)} \\
  \hline 
  {I}_{N \times N} & {0}_{N \times N} & \cdots & {0}_{N \times N}  \\
  \hline 
  {0}_{N \times N} & {I}_{N \times N} & \cdots & {0}_{N \times N}\\
  \hline
  \vdots & \vdots &  \vdots & \vdots \\
  \hline 
  {0}_{N \times N} & {0}_{N \times N} & \cdots & {0}_{N\times N}
\end{array}\right] 
\end{align}
is a block matrix of size $NL\times NL$. It has $L^{2}$ blocks, each of which is a square matrix of size $N$. For two matrices $A$ and $B$, each of size $N \times N$, we denote $A \circ B$ to be the Hadamard (or Schur) product (see section~\ref{sec:preliminaries}) of $A$ and $B$. Furthermore, $e_{N,NL}$ is a column vector of length $NL$ with ones in the first entries $N$ and zeros in the remaining entries. We call $\mathbf{M}_{N}^{(k)}$ the ``{indicator matrix}'' of size $N$ for the time instant $t-k$ (for brevity, we omit writing $t$ in the notation here), and it is defined as follows:
\begin{align}\label{eqn:bold_M_matrix}
[\mathbf{M}_{N}^{(n)}]_{ij}:= \begin{cases}
1 & \textrm{if} \quad t-1-n \in \mathcal{M}_{ji}^{(t-1)} \\
0 & \textrm{otherwise}.
\end{cases}    
\end{align}
Note that, for ease of notation we have written $J_{N,\mathbf{M}_{N}^{(1)},\cdots, \mathbf{M}_{N}^{(L)}}^{(L)}$ as $J_{N}^{(L)}$. In the next section, we discuss the stability analysis of \eqref{eqn:X}. 

%With this definition in place, we have the following:
%\begin{enumerate}
%    \item Each indicator matrix is a binary matrix: $\mathbf{M}_N^{(k)} \in \{0,1\}^{N\times N}, \text{ for }
%\qquad k=1,\ldots,L$.

%\item Since each memory set
%\(
%\mathcal{M}_{ji}^{(t)}
%\)
%contains exactly \(M\) elements, we have
%\[
%\sum_{k=1}^{L}
%[\mathbf{M}_N^{(k)}]_{ij}
%= M, \text{ for all} \hspace{0.2cm} i, j \in \{1,\cdots,N\}.
%\]
%Equivalently, $\sum_{k=1}^{L}
%\mathbf{M}_N^{(k)} = M\,\mathbf{1}_N\mathbf{1}_N^{\mathbf T}$.
%\item Hadamard decomposition of \(M S_N\): for $k=1,\ldots,L$, define $B_k \coloneqq S_N\circ \mathbf{M}_N^{(k)}.$ Then, from the previous identity if follows that
%\[
%\sum_{k=1}^{L} B_k
%=
%M S_N.
%\]

%\item Spectral radius bound: Since
%\(
%[\mathbf{M}_N^{(k)}]_{ij}\in\{0,1\},
%\)
%we have $0 \le B_k \le S_N$, for all $k=1,\ldots,L$. Perron--Frobenius theory \cite{RAH-CRJ:12} yields that for every $k \in \{1,\cdots,L\}$,
%\[ \rho(B_k) \le \rho(S_N) = 1. \]

%\item For every pair $(i,j)$,
%\[
%\sum_{k=1}^{L}
%(B_k)_{ij}
%=
%\sum_{k=1}^{L}
%s_{ij}
%[\mathbf{M}_N^{(k)}]_{ij}
%=
%M s_{ij}.
%\]
%Thus, due to time-homogeneity of memory sets, each edge weight $s_{ij}$ is replicated exactly $M$ times
%across the collection $\{B_k\}_{k=1}^{L}$.
%\end{enumerate}

\section{Stability and Equilibrium Analysis} \label{sec:Analysis}
We begin by establishing a spectral property of the matrix $J_{N}^{(L)}$.
\begin{lemma}\label{lem:spectral_radius}
    Suppose $J_{N}^{(L)}$ is as defined above. Then, $\rho(J_{N}^{(L)})) < 1$.
\end{lemma}
\textit{Proof:}
    Suppose $\lambda \in \sigma(J_N^{(L)})$.  Then, $J_N^{(L)} v = \lambda v$. Write $v=(v_1, \dots, v_L)^{\mathbf{T}}$ such that $v_i \in \mathbb{R}^N$. Then, for the homogeneous case
\begin{align*}
\lambda v_1 = c_{2, M} \sum_{l=1}^L (S_N \circ \mathbf{M}_N^{(l)}) v_l;\quad  
v_1 = \lambda v_2;\quad 
v_2 = \lambda v_3 \quad \cdots \quad 
v_{L-1} = \lambda v_{L}.
\end{align*}
On further solving this yields
\[ \lambda v_1 = c_{2, M} \sum_{l=1}^L \lambda^{-(l-1)} (S_N \circ \mathbf{M}_N^{(l)}) v_1  \]
Multiplying by $\lambda^{L-1}$ we get,
 \begin{align} \label{eq:chareq} \lambda^L v_1 = c_{2, M} \sum_{l=1}^L \lambda^{L-l} (S_N \circ \mathbf{M}_N^{(l)} )v_1.  \end{align}
Suppose $i=\arg\max |v_1(k)|$, where $v_1(k)$ denotes the $k$th co-ordinate of the vector $v_1$. Evaluating the norm on the $i$th row on both sides of \eqref{eq:chareq}, we get the following.
 \begin{align*} 
 \vert \lambda \vert^L \vert v_1(i) \vert  \leq \sum_{l=1}^L  |c_{2, M}| \vert \lambda \vert^{L-l}  \sum_{j=1}^N  \vert [S_N \circ \mathbf{M}_N^{(l)}]_{ij} \vert | v_1(j) | \leq  |c_{2, M}| \vert \lambda \vert^{L-l} \sum_{l=1}^L \sum_{j=1}^N  \vert [S_N \circ \mathbf{M}_N^{(l)}]_{ij} \vert | v_1(i) |.
 \end{align*}
 Thus,
\begin{align}\label{eq:final_contr} \vert \lambda \vert^L \vert \leq \vert \lambda \vert^{L-l} \sum_{l=1}^L R_i^{(k)}, \end{align}
where $R_i^{(k)}$ is the row-sum of the $i$-th row in the $k$-th block. Since $\sum_{l=1}^{L}R_i^{(l)}\le M$, if $|\lambda|\ge 1$, then
$|\lambda|^{L-l}\le |\lambda|^{L}$ for all $l\in\{1,\ldots,L\}$.
Hence,
\[
|\lambda|^{L}
\le
|c_{2,M}|\,|\lambda|^{L}\sum_{l=1}^{L}R_i^{(l)}
\le
M|c_{2,M}|\,|\lambda|^{L},
\]
which contradicts the assumption $M|c_{2,M}|<1$. Therefore,
$|\lambda|<1$, and hence $\rho(J_{N}^{(L)})<1$.
$\null\nobreak\hfill\ensuremath{\square}$

We are now ready to characterize the equilibrium of the dynamics in~\eqref{eqn:P_eqn}.
\begin{theorem}[Existence, uniqueness and stability of the equilibrium point]\label{thm:fixed_pt}
    The unique and stable fixed point of the recursion in \eqref{eqn:P_eqn} is given by,
\begin{align}\label{eqn:homo_fixed_point}
P^{(*)} = \frac{c_{1,M}}{1-c_{2,M}M}\mathbf{1}_{N},
\end{align}
where $P^{(*)}: = (P_{1}^{(*)},\cdots,P_{N}^{(*)})^{\mathbf{T}}$.
\end{theorem}

%\textit{Proof:} %We begin by observing that given $|c_{2,M}|<\frac{1}{M}$, the row-sums of the first block of $J_{N}^{(L)}$ are strictly less than one. Therefore, $({I}_{NL \times NL}-J_{N}^{(L)})$ is an invertible matrix. 
 
\textit{Proof:}
By Lemma~\ref{lem:spectral_radius}, we have $\rho(J_{N}^{(L)})<1$. Hence,
$(I_{NL\times NL}-J_{N}^{(L)})$ is invertible, and \eqref{eqn:X}
admits a unique fixed point given by
\begin{align}\label{eqn:PL_fixed_pt}
\widetilde{P}^{(*,L)}
=
c_{1,M}(I_{NL\times NL}-J_{N}^{(L)})^{-1}e_{N,NL}.
\end{align}
Moreover, $\rho(J_{N}^{(L)})<1$ implies that this fixed point is
globally asymptotically stable.

Since $L$ is finite and
$\widetilde{P}_{t}^{(L)}\to\widetilde{P}^{(*,L)}$ as $t\to\infty$,
it follows that
\[
\widetilde{P}^{(*,L)}
=
(P^{(*)},\cdots,P^{(*)})^{\mathbf T},
\]
where
$P^{(*)}=(P_{1}^{(*)},\cdots,P_{N}^{(*)})^{\mathbf T}$
is the fixed point of \eqref{eqn:P_eqn}. Therefore, taking the limit
as $t\to\infty$ in \eqref{eqn:P_eqn}, we obtain
\[
P_i^{(*)}
=
c_{1,M}
+
c_{2,M}M\sum_{j=1}^{N}s_{ij}P_j^{(*)},
\]
which can be written in matrix form as
\begin{align}\label{eqn:P_star_matrix}
P^{(*)}
=
c_{2,M}MS_NP^{(*)}
+
c_{1,M}\mathbf1_N.
\end{align}

Since $S_N$ is a row-stochastic matrix,
\[
(I_{N\times N}-c_{2,M}MS_N)\mathbf1_N
=
(1-c_{2,M}M)\mathbf1_N.
\]
Further, the row sums of $c_{2,M}MS_N$ are strictly less than one.
Hence $(I_{N\times N}-c_{2,M}MS_N)$ is invertible, and
\[
(I_{N\times N}-c_{2,M}MS_N)^{-1}\mathbf1_N
=
\frac{1}{1-c_{2,M}M}\mathbf1_N.
\]
Therefore,
\[
P^{(*)}
=
\frac{c_{1,M}}{1-c_{2,M}M}\mathbf1_N.
\]
$\null\nobreak\hfill\ensuremath{\square}$

%%%%%%%%%%%%%%%%%%%%%%%%%%%%%%%%%%%%%%%%%%%%%%%%%%%%%%%%%%%%%%%%%%%%%%%%%%%% End of Proof of Theorem 1  %%%%%%%%%%%%%%%%%%%%%%%%%%%%%%%%%%%%%%%%%%%%%%%%%%%%%%%%%%%%%%%%%%%%%%%%%%%%%%%%%
We now use block matrix inversion formula~\cite[section 0.7.3]{RAH-CRJ:12} to find a closed-form formula for $({I}_{NL\times NL}- J_{N}^{(L)})$. To this end, we define the following matrices:
\begin{align*}
&B_{1}  = \left[{I}_{N \times N}-c_{2,M}S_{N} \circ \mathbf{M}_{N}^{(1)}\right] \quad \quad 
B_{2}  = \left[ 
\begin{array}{c|c|c} 
   -c_{2,M}S_{N}\circ \mathbf{M}_{N}^{(2)}& \cdots & -c_{2,M}S_{N}\circ \mathbf{M}_{N}^{(L)}
\end{array}\right]\\ 
&B_{3}  = \left[ 
\begin{array}{c} 
  -{I}_{N\times N}  \\
  \hline 
  {0}_{N\times N} \\
  \hline
  \vdots \\
  \hline 
  {0}_{N \times N}
\end{array}\right]\quad \quad B_{4}  = \left[ 
\begin{array}{c|c|c|c} 
   {I}_{N \times N} & {0}_{N \times N} & \cdots & {0}_{N \times N}  \\
  \hline 
    -{I}_{N \times N} & {I}_{N \times N} & \cdots & {0}_{N \times N}\\
  \hline
   \vdots & \vdots & \vdots & \vdots \\
  \hline 
   {0}_{N \times N} & \cdots & -{I}_{N \times N} & {I}_{N \times N}
\end{array}\right].
\end{align*}
We can now compute:
\begin{align}\label{eqn:I-J_inverse}
&(I_{NL\times NL} -  J_{N}^{(L)})^{-1} = \begin{bmatrix}
    B_{1} & B_{2}\\
    B_{3} & B_{4}
\end{bmatrix}^{-1} = \begin{bmatrix}
 B_{1}^{-1} + B_{1}^{-1}B_{2}((I-J)/B_{1})^{-1}B_{3}B_{1}^{-1} & -B_{1}^{-1}B_{2}((I-J)/B_{1})^{-1}\\
 -((I -J)/B_{1})^{-1}B_{3}B_{1}^{-1} & ((I-J)/B_{1})^{-1}
\end{bmatrix},
\end{align}
where $(I-J)/B_{1} :=((I_{NL\times NL}-J_{N}^{(L)})/B_{1}) = B_{4} - B_{3}B_{1}^{-1}B_{2}$ is the Schur complement~\cite[section 0.8.5]{RAH-CRJ:12}.

For ease of notation, we will denote $p^{(*)}:= c_{1,M}/(1-c_{2,M}M)$ hereafter. A crucial takeaway for \eqref{eqn:homo_fixed_point} is that $P^{(*)}$ does not depend on the network structure and interaction weights $s_{ij}$'s which occurs due to homogeneity of the model ($c_{1,M}$ and $c_{2,M}$ are same for all the individuals in the network).  

We now briefly discuss the non-homogeneous case and other properties of the model in the following remarks:
\begin{remark}\label{rem:C1_zero} As discussed earlier, for the case $c_{1,M}=0$ in \eqref{eqn:ratio_X}, there is a nonzero probability that $Z_{i,n}=0$ for all $n \in \mathcal{M}_{ji}^{(t)}$, and for all $i,j \in \{1,\cdots,N\}$ at some time instant $t$. Such an occurrence would imply $P_{i}(t') =0$ for all $t'>t$, for all $i \in \{1,\cdots,N\}$. Therefore, $P^{(*)}= \mathbf{0}_{N}$ for this case, which can also be obtained  from the fixed point equation $\eqref{eqn:homo_fixed_point}$. 
\end{remark}
\begin{remark}\label{rem:non_homo}(Non-homogeneous case). Note that the constants $c_{1,M}$ and $c_{2,M}$ in \eqref{eqn:ratio_X} are the same for all the individuals in the network, which corresponds to the homogeneity in our model. For the non-homogeneous case, the entries of the fixed point $P^{(*)}$ are node-dependent and difficult to compute. However, similar to the homogeneous case, we can show the existence of the unique fixed point for the corresponding recursion obtained in $\widetilde{P}^{(*,L)}$. We generalize our opinion dynamics model to a non-homogeneous form as follows:
\begin{align}\label{eqn:X_non_homo}
X_{ij,t} = c_{1,M}^{(i)} + c_{2,M}^{(i)}\hspace{-0.3cm}\sum_{n\in \mathcal{M}_{ji}^{(t)}}\hspace{-0.3cm}Z_{j,n},   
\end{align}
where $\widetilde{c}_{1,M}:= (c_{1,M}^{(1)},c_{1,M}^{(2)},\cdots,c_{1,M}^{(N)})^{\mathbf{T}}$, and $\widetilde{c}_{2,M}:= (c_{2,M}^{(1)},c_{2,M}^{(2)},\cdots,c_{2,M}^{(N)})^{\mathbf{T}}$, with following two assumptions:
\begin{itemize}\setlength{\itemsep}{6pt}
\item[(1)] $\max (0,-c^{(i)}_{2,M}K)\leq c^{(i)}_{1,M} \leq \min (1,1-c^{(i)}_{2,M}K)$ for $K \in \{0,1,\cdots,M\}$ and $i\in \{1,\cdots,N\}$.
\item[(2)] $ c^{(i)}_{2,M} \in (-1/M, 1/M) - \{0\}$ for all $i \in \{1,\cdots,N\}$.
\end{itemize}
Similar to the previous analysis, we obtain the following recursion in $\widetilde{P}_{t}^{(L)}$ here:
\begin{align}\label{eqn:recursion_non_homo}
\widetilde{P}_{t}^{(L)} = J_{N}^{(L)}\widetilde{P}_{t-1}^{(L)} + [\hspace{0.1cm}\widetilde{c}_{1,M}\hspace{0.1cm}|\hspace{0.1cm}\mathbf{0}_{N(L-1)}].
\end{align}
Here, $[\hspace{0.1cm}\mathbf{a}\hspace{0.1cm}|\hspace{0.1cm}\mathbf{b}\hspace{0.1cm}]$ is concatenation of two column vectors $\mathbf{a}$ and $\mathbf{b}$, and  
\begin{align}\label{eqn:J_matrix_non_homo}
 J_{N}^{(L)}= \left[ 
\begin{array}{c|c|c|c} 
  C_{2,M}^{(N)}\circ S_{N}\circ \mathbf{M}_{N}^{(1)}& C_{2,M}^{(N)}\circ S_{N}\circ \mathbf{M}_{N}^{(2)}&  \cdots & C_{2,M}^{(N)}\circ S_{N}\circ \mathbf{M}_{N}^{(L)} \\
  \hline 
 {I}_{N\times N} & {0}_{N \times N} & \cdots & {0}_{N \times N}  \\
  \hline 
  {0}_{N \times N} & {I}_{N \times N} & \cdots & {0}_{N \times N}\\
  \hline
  \vdots & \vdots & \vdots & \vdots \\
  \hline 
  {0}_{N \times N} & {0}_{N \times N} & \cdots & {0}_{N \times N}
\end{array}\right] 
\end{align}
where $C_{2,M}^{(N)} := \widetilde{c}_{2,M}\mathbf{1}_{N}^{\mathbf{T}}$ is a matrix of size $N$. The invertibility of $(I_{NL\times NL}-J_{N}^{(L)})$ guarantees the existence of a unique fixed point for \eqref{eqn:recursion_non_homo}:
\begin{align}\label{eqn:fixed_pt_vector_non_homo}
\widetilde{P}^{(*,L)} = (I_{NL\times NL}-J_{N}^{(L)})^{-1}[\hspace{0.1cm}\widetilde{c}_{1,M}\hspace{0.1cm}|\hspace{0.1cm}\mathbf{0}_{N(L-1)}]. 
\end{align}
Unlike the homogeneous case, the fixed point in \eqref{eqn:fixed_pt_vector_non_homo} depends on the entries of the matrix $S_{N}$ (i.e., the network structure and edge weights). However, for the non-homogeneous case, it is challenging to obtain a closed-form formula for this fixed point. For instance, the fixed point $\widetilde{P}^{(*,1)}$ for a $2$-node network equipped with the interaction matrix $S_{N} = \begin{bmatrix}
    s_{11} & s_{12}\\
    s_{21} & s_{22}
\end{bmatrix}$ is given by: 
\begin{align}\label{eqn:fix_pt_nonhomo_twonodes}
\widetilde{P}^{(*,1)} = \begin{bmatrix}
    \frac{(1-c_{2,1}^{(2)}s_{22})c_{1,1}^{(1)} + c_{2,1}^{(2)}c_{1,1}^{(2)}s_{12}}{(1-c_{2,1}^{(1)}s_{11})(1-c_{2,1}^{(2)}s_{22})-c_{2,1}^{(1)}c_{2,1}^{(2)}s_{12}s_{21}}\\\\
    \frac{(1-c_{2,1}^{(1)}s_{11})c_{1,1}^{(2)} + c_{2,1}^{(1)}c_{1,1}^{(1)}s_{21}}{(1-c_{2,1}^{(1)}s_{11})(1-c_{2,1}^{(2)}s_{22})-c_{2,1}^{(1)}c_{2,1}^{(2)}s_{12}s_{21}}
\end{bmatrix}.    
\end{align}
One can verify by direct differentiation that each entry of~\eqref{eqn:fix_pt_nonhomo_twonodes} is increasing in $M$ 
(by replacing $c_{2,1}$ by $M \cdot c_{2,M}$ and differentiating) under the 
parameter constraints of the model. For the homogeneous case, the monotonicity follows directly from~\eqref{eqn:homo_fixed_point}. We formalize this observation as the following conjecture for the non-homogeneous setting: 

\begin{conjecture}\label{thm:conjecture}
(Monotonicity of Fixed Point in Memory for Non-Homogeneous Networks).
Consider the non-homogeneous model~\eqref{eqn:X_non_homo} with parameters
$c_{2,M}^{(i)} \in (0, 1/M)$ and
$0< c_{1,M}^{(i)} < \min(1,1-c_{2,M}^{(i)}M)$
for all $i \in \{1,\cdots,N\}$, and a fixed network $\mathcal{G}_N$
with interaction matrix $S_N$. Then, with all parameters other than the memory size $M$ held fixed, the fixed point $\widetilde{P}^{(*,L)}$ in \eqref{eqn:fixed_pt_vector_non_homo} is strictly increasing with $M$ for each $i \in \{1,\ldots,N\}$.
\end{conjecture}

The general case for arbitrary $N$ and $M$ remains open and constitutes a direction for future work. 
\end{remark}
\begin{remark}\label{rem:Markov}(Markov chain analysis for our model). Given a time instant $t>T$, we define the random vector $Z_{t}:=(Z_{1,t},\cdots,Z_{N,t})$ for the network $\mathcal{G}_{N}$ equipped with the relative bias update given by \eqref{eqn:X_non_homo}. Given that $|\mathcal{M}^{(t)}_{ji}|=M$ for all $i,j\in \{1,\cdots,N\}$ (this follows from time-homogeneity of memory sets), the random vector $Z_{t}$ is a $L$-order Markov chain, where $L:=\max_{n}\{n \hspace{0.1cm}|\hspace{0.1cm} t-n \in \cup_{i,j=1}^{N}\mathcal{M}_{ji}^{(t)} \hspace{0.1cm} \textrm{for all} \hspace{0.1cm} t>T\}$ is the memory depth. We can write the transition probabilities for this Markov chain using \eqref{eqn:draw_S} and \eqref{eqn:X_non_homo}:
\begin{align}\label{eqn:tp}
P[Z_{t} =a_{L+1}|Z_{t-1}=a_{L},\cdots ,Z_{t-L}=a_{1}]= \prod_{i=1}^{N}\Bigg((2a_{i,L+1}-1)\sum\limits_{j=1}^{N}s_{ij}(c_{1,M}^{(i)} + c_{2,M}^{(i)}\hspace{-0.4cm}\sum_{t-k\in \mathcal{M}_{ji}^{(t)}}\hspace{-0.4cm}a_{j,k}) +(1-a_{i,L+1})\Bigg),
\end{align}
where $a_{l}=(a_{1,l},\cdots,a_{N,l})\in \{0,1\}^{N}$, for $l \in \{1,\cdots, L+1\}$.
Note that in \eqref{eqn:tp} we write the transition probabilities of the underlying Markov process for the non-homogeneous case. Furthermore, the presence of strictly positive drift factors $c_{1,M}^{(i)}$'s makes the underlying Markov process irreducible and aperiodic which guarantees the existence of a unique stationary distribution (finite state space also implies positive recurrence, therefore this Markov chain is ergodic). However, due to the complex structure of the transition probability matrix, it is difficult to obtain a closed-form expression for the stationary distribution. An important insight here is that the fixed point for \eqref{eqn:fixed_pt_vector_non_homo} is a function of the stationary distribution for the concatenated Markov process $\widetilde{Z}_{t} :=(Z_{t},\cdots, Z_{t+L-1})$. To see this, let the corresponding stationary distribution be $\Pi$ with entries given by $\pi(a_{1},\cdots,a_{L})$ (Note that, $a_{l}$'s are defined in \eqref{eqn:tp} and the stationary vector has $2^{NL}$ entries). Then, the following holds: 
\begin{align}
\lim_{n \to \infty}P(Z_{i,t}=1) \hspace{-0.1cm} = \hspace{-0.5cm}\sum_{\substack{a_{i,l}=1\\l \in \{1,\cdots,L\}}}\hspace{-0.4cm}\pi(a_{1},\cdots,a_{L}) = \widetilde{P}^{(*,L)}_{i},
\end{align}
where $\widetilde{P}^{(*,L)}_{i}$ is the $i$th entry of the fixed point vector in \eqref{eqn:fixed_pt_vector_non_homo}. A detailed analysis of a related Markovian process has been carried out for interacting networks of finite-memory Pólya urns in \cite{SS-FA-BG:22}. Several of the corresponding results can be extended to our model by adapting similar analytical techniques.    
\end{remark}
\begin{remark}\label{rem:Friedman}(Opinion Dynamics model as an interacting network of finite memory Friedman urns).
We now discuss a special case of our model given by an interacting $N$-node network of finite memory Friedman urns. Each individual $i$ in the network is equipped with time-homogeneous memory sets $\mathcal{M}_{ji}^{(t)}$ for $j \in \{1,\cdots,N\}$ and a two-color Friedman urn (we refer the reader to~\cite{BF:49} for a description of classical Friedman urns). The memory here refers to removal of all the balls that were added to the urn at a time instant which does not belong to the corresponding memory set. To this end, we let $X_{ij,t}$ to be the ratio of red balls in urn $j$ relative to urn $i$ at time $t$, then using \eqref{eqn:ratio_X}:
\begin{align}\label{eqn:ratio}
X_{ij,t} = \frac{\rho + (\delta-\phi)\sum_{n \in \mathcal{M}_{ji}^{(t)}}Z_{j,n}+M\phi}{1+M(\delta+\phi)},\end{align}
where $\rho$ is the initial ratio of red balls in all the urns, and the urn scheme for Friedman urns is given by $
\begin{bmatrix}
    \delta & \phi\\
    \phi   & \delta
\end{bmatrix}$.
 On comparing \eqref{eqn:ratio} with \eqref{eqn:ratio_X}, we obtain $c_{1,M} = (\rho + M\phi)/(1+M(\delta+\phi))$ and $c_{2,M} = (\delta-\phi)/(1+M(\delta+\phi))$. Note that, the expressions for $c_{1,M}$ and $c_{2,M}$ here are consistent with the assumptions established for our model. Furthermore, the fixed point vector for this case is $P^{(*)} = (\rho + M\phi)/(1+2M\phi)\mathbf{1}_{N}$. Note that, taking $M\to \infty$ in this expression gives $P^{(*)}= \frac{1}{2}\mathbf{1}_{N}$, which happens to be the limiting distribution for a classical two-color Friedman urn (see \cite{DAF:65} for details). Thus the finite memory version of a Friedman urn can be used to obtain an alternate proof for asymptotics of the infinite memory case (i.e., the classical urn). However, such an insight on asymptotics is not obtained for interacting network of finite memory P\'{o}lya urns \cite{SS-FA-BG:22}, where $P^{(*)} = \rho \mathbf{1}_{N}$ for the homogeneous case. For a detailed study of asymptotic behavior of classical two-color P\'{o}lya and Friedman urns we refer the readers to \cite{BF:49,HMM:09}. 
\end{remark}
\begin{remark}(Behavioral Interpretation of Relative Bias and the Update Rule).
The notion of \emph{relative bias} $X_{ij,t}$ captures the disposition of agent~$i$ 
toward the opinion of agent~$j$ as perceived through $i$'s memory of $j$'s past expressed opinions. 
This formulation is natural in several practical settings, such as 

\begin{itemize}

\item \textit{Communication networks with per-link memory}. 
Consider a network of node exchanging binary messages (e.g., content-approval votes, congestion signals). 
Node~$i$ may store different amounts of history from different neighbors due to heterogeneous buffer 
capacities or link qualities, leading to the memory sets $\mathcal{M}_{ji}^{(t)}$ being 
indexed by the ordered pair $(j,i)$. The relative bias $X_{ij,t}$ then models node~$i$'s ``trust'' 
in the typical message from node~$j$, updated using only the messages $i$ has actually received from $j$. 
Self-influence enters only through the weight $s_{ii}$ in~(2), reflecting the degree to which agent~$i$ 
relies on its own internal state versus incoming information.

\item \textit{Social influence with selective attention.} In social networks, individuals often form impressions of others based on a recent window of observed 
posts or statements, while their own opinion at any given moment is a weighted aggregate of these 
impressions across their social contacts. The update rule~(3) formalizes this: the reinforcement 
$c_{2,M}\sum_{n \in \mathcal{M}_{ji}^{(t)}} Z_{j,n}$ captures how much the recent behavior of 
agent~$j$, as remembered by $i$, shifts $i$'s disposition toward $j$, while the drift $c_{1,M}$ 
prevents collapse to absorbing states (see Remark~\ref{rem:Markov}). The linearity of the update is a modeling 
choice that ensures analytical tractability while preserving the essential asymmetry: 
$X_{ij,t}$ and $X_{ji,t}$ evolve independently, consistent with real-world asymmetric influence.
\end{itemize}
\end{remark}

\section{Addition of Bots to the network}\label{sec:bots}
In this section, we extend our opinion dynamics model as described in section~\ref{sec:model} to accommodate the addition of bots to the network. These bots do not change their assigned bias but provide a constant shift to the connected agents in the network towards their bias. We denote $\mathcal{G}^{(\mathcal{B})}_{N}$ to be an interacting $(N+|\mathcal{B}|)$-node network integrated with $N$ individuals and a set of bots $\mathcal{B}$. We assume that every bot influences at least one regular individual in the network (i.e., the network including the bots is connected). Every bot $k \in \mathcal{B}$ has an associated strength given by $\eta_{k}\in (0,1]$, which is the constant bias of the bot. To encode the influence of bots on individuals in the network, we modify the expression for relative biases in \eqref{eqn:ratio_X} as follows:

\begin{align}\label{eqn:ratio_X_bots}
X_{ij,t}^{(\mathcal{B})} := c_{1,M} + c_{2,M}\hspace{-0.3cm}\sum_{n \in \mathcal{M}_{ji}^{(t)}}\hspace{-0.2cm}Z_{j,n}^{(\mathcal{B})} \hspace{0.2cm} \textrm{for all} \hspace{0.2cm} i \in \{1,2,\cdots,N\}, 
\end{align}
where $Z_{j,t}^{(\mathcal{B})}$ is the expressed opinion of individual $j$ in the presence of bot set $\mathcal{B}$ at time $t$. We extend the expression for expressed opinion in \eqref{eqn:draw_S} to account for the bot influence on individuals as follows:

\begin{align}\label{eqn:drawing_bots}
Z^{(\mathcal{B})}_{i,t} = \begin{cases}
    1 & \textrm{w.p.} \quad \sum\limits_{j=1}^{N}\widetilde{s}_{ij}X^{(\mathcal{B})}_{ij,t-1} + \sum\limits_{k \in \mathcal{B}}\widetilde{s}_{ik}\eta_{k}\\
    0 & \textrm{w.p.} \quad \sum\limits_{j=1}^{N}\widetilde{s}_{ij}(1-X^{(\mathcal{B})}_{ij,t-1}) + \sum\limits_{k \in \mathcal{B}}\widetilde{s}_{ik}(1-\eta_{k}), 
\end{cases}
\end{align}
for all $i \in \{1,\cdots,N\} \cup \mathcal{B}$. Here, $\widetilde{s}_{ij}$ is  the non-negative influence of $j$ on $i$, and is the $ij$th entry of the interaction matrix $\widetilde{S}^{(\mathcal{B})}_{N}$ which can be written as the following block matrix:
\begin{align}\label{eqn:S_bots}
\widetilde{S}^{(\mathcal{B})}_{N} :=\left[ 
\begin{array}{c|c} 
  \begin{array}{c} \widetilde{S}_{N}\end{array} & \widetilde{S}_{\mathcal{B}}\\ 
  \hline 
  {0}_{|\mathcal{B}|\times N} & {I}_{|\mathcal{B}| \times |\mathcal{B}|}  
\end{array} \right].
\end{align}
Here, the $N \times N$ matrix $\widetilde{S}_{N}$ consists of interaction weights among $N$ individuals of the network and the $N \times |\mathcal{B}|$ matrix $\widetilde{S}_{\mathcal{B}}$ determines the interaction weight of bots on individuals. The $2 \times 1$ block of \eqref{eqn:S_bots} is zero because individuals do not influence the opinion of bots. Similarly, $2 \times 2$ block of \eqref{eqn:S_bots} is identity due to stubbornness of the bots, i.e, the expressed opinion of a bot at any time instant only depends on its own strength. In order to ensure that $Z_{i,t}^{(\mathcal{B})}$ is a well-defined indicator function, we assume $\sum_{j=1}^{N}\widetilde{s}_{ij} + \sum_{k \in \mathcal{B}}\widetilde{s}_{ik}=1$ for all $i \in \{1,\cdots,N\}$. Furthermore, the expressed opinion for bots at any time step as given by \eqref{eqn:drawing_bots} is:
\begin{align}\label{eqn:drawing_for_bots}
Z^{(\mathcal{B})}_{k,t} = \begin{cases}
    1 & \textrm{w.p.} \quad \eta_{k}\\
    0 & \textrm{w.p.} \quad (1-\eta_{k}) 
\end{cases}
\quad \textrm{for all} \hspace{0.2cm} k \in \mathcal{B},
\end{align}
which confirms that each bot $k \in \mathcal{B}$ expresses opinions exclusively according to its own bias/strength $\eta_{k}$. For this reason, we have defined the relative biases in \eqref{eqn:ratio_X_bots} only for individuals $i \in \{1,\cdots,N\}$ and not for the bots, in other words, bots act as fixed external influence on the network.

Similar to the previous section, we obtain the corresponding discrete-time  dynamical system for the beliefs  $P^{(\mathcal{B})}_{i}(t):= P(Z^{(\mathcal{B})}_{i,t}=1)$ for all $i \in \{1,\cdots,N\}$ as follows:
\begin{align}\label{eqn:P_calc_bots}
P^{(\mathcal{B})}_{i}(t) &= E [E[Z^{(\mathcal{B})}_{i,t}|\mathcal{F}^{(\mathcal{B})}_{t-1}]] = E\bigg[\hspace{-0.09cm}\sum_{j=1}^{N}\widetilde{s}_{ij}X^{(\mathcal{B})}_{ij,t-1} \hspace{-0.1cm}+ \sum_{k \in \mathcal{B}}\widetilde{s}_{ik}\eta_{k}\hspace{-0.09cm}\bigg] = \sum_{j=1}^{N}\widetilde{s}_{ij}E\bigg[c_{1,M} + c_{2,M}\hspace{-0.3cm}\sum_{n \in \mathcal{M}_{ji}^{(t-1)}}\hspace{-0.3cm}Z^{(\mathcal{B})}_{j,n}\bigg] + \sum_{k \in \mathcal{B}}\widetilde{s}_{ik}\eta_{k} \nonumber\\
&= c_{2,M}\sum_{j=1}^{N}\sum_{n \in \mathcal{M}^{(t-1)}_{ji}}\hspace{-0.3cm}\widetilde{s}_{ij}P^{(\mathcal{B})}_{j}(n) + c_{1,M}\sum_{j=1}^{N}\widetilde{s}_{ij} + \sum_{k \in \mathcal{B}}\widetilde{s}_{ik}\eta_{k},
\end{align}
where, $\mathcal{F}^{(\mathcal{B})}_{t-1}$ is the smallest sigma-algebra defined on the random variables $\{Z^{(\mathcal{B})}_{1,n},Z^{(\mathcal{B})}_{2,n},\cdots,Z^{(\mathcal{B})}_{N,n}\}_{n \in \bigcup\limits_{j,k=1}^{N}\hspace{-0.02cm}\mathcal{M}_{jk}^{(t)}}$. To further obtain the matrix equation corresponding to \eqref{eqn:P_calc_bots}, we denote $P^{(\mathcal{B})}(t) := (P^{(\mathcal{B})}_{1}(t),\cdots, P^{(\mathcal{B})}_{N}(t))^{\mathbf{T}}$ and $\widetilde{P}_{t}^{(\mathcal{B},L)} := (P^{(\mathcal{B})}(t),\cdots,P^{(\mathcal{B})}(t-L+1))^{\mathbf{T}}$, where $L$ is the memory depth of the model as defined in \ref{sec:model}:
\begin{align}\label{eqn:ds_N_bots}
 \widetilde{P}^{(\mathcal{B},L)}_{t} =  J^{(\mathcal{B},L)}_{N}\widetilde{P}^{(\mathcal{B},L)}_{t-1} + C^{(\mathcal{B},L)}_{N},
 \end{align}
 where $J^{(\mathcal{B},L)}_{N}\in \mathbb{R}^{NL \times NL}$ and $C^{(\mathcal{B},L)}_{N}\in \mathbb{R}^{NL \times 1}$ are the following matrices:
\begin{align}\label{eqn:J_bots}
J_{N}^{(\mathcal{B},L)}= \left[ 
\begin{array}{c|c|c|c} 
  c_{2,M}\widetilde{S}_{N}\circ \mathbf{M}_{N}^{(1)}& c_{2,M}\widetilde{S}_{N}\circ \mathbf{M}_{N}^{(2)} & \cdots & c_{2,M}\widetilde{S}_{N}\circ \mathbf{M}_{N}^{(L)} \\
  \hline 
  {I}_{N\times N} & {0}_{N\times N} & \cdots & {0}_{N\times N}  \\
  \hline 
  {0}_{N\times N} & {I}_{N\times N} & \cdots & {0}_{N\times N}\\
  \hline
  \vdots & \vdots & \vdots & \vdots \\
  \hline 
  {0}_{N\times N} & {0}_{N\times N} & \cdots & {0}_{N\times N}
\end{array}\right]
\end{align}
and \begin{align}\label{eqn:C_bots}
[C^{(\mathcal{B})}_{N}]_{i\times 1}=
\begin{cases}
    c_{1,M}\sum\limits_{j=1}^{N}\widetilde{s}_{ij} + \sum\limits_{k \in \mathcal{B}}\widetilde{s}_{ik}\eta_{k}, & i \in \{1,\cdots,N\}\\\\
    0 & \textrm{otherwise.}
\end{cases}
\end{align}
We can use an argument similar to the previous section to show the existence of a unique fixed point (denoted by $\widetilde{P}^{(\mathcal{B},L,*)} = (P^{(\mathcal{B},*)},\cdots,P^{(\mathcal{B},*)})$ for \eqref{eqn:ds_N_bots}). In the previous section, we deduced that for a homogeneous network, the fixed point $\widetilde{P}^{(L,*)}$ has entries $p^{(*)}$ (see \eqref{eqn:homo_fixed_point}). To understand the influence of bots on the network, we now compute $P^{(\mathcal{B},*)}$ for different networks and look the deviation of its entries from $p^{(*)}$, i.e., how much the fixed point shifts with addition of bots. 

\begin{example}\label{eg:two_nodes_one_bot}(A complete network on two nodes with one bot as shown in Fig.~\ref{fig:two_nodes_one_bot})
\begin{figure}[h!]
\begin{center}
    \begin{tikzpicture}[->, >=stealth, node distance=2.5cm, thick]
        % Define nodes
        \node[circle, draw, minimum size=1cm] (1) {$1$};
        \node[circle, draw, minimum size=1cm, right of=1] (2) {$2$};
        \node[circle, draw, minimum size=1cm, below of=2, draw=red, fill=red!20] (B) {$B$};

        % Define directed edges with labels
        \path[->] (1) edge[bend left] node[above] {$\widetilde{s}_{21}$} (2);
        \path[->] (2) edge[bend left] node[above] {$\widetilde{s}_{12}$} (1);
       \path[->,thick, black]
    (1) edge [out=195, in=165, loop] node[left]{$\widetilde{s}_{11}$} (1);
    \path[->,thick, black]
    (2) edge [out=-10, in=20, loop] node[right]{$\widetilde{s}_{22}$} (2);

        % Red edges directed from B towards flexible nodes
        \path[->, red, thick] (B) edge[bend left] node[left, black] {$\widetilde{s}_{1B}$} (1);
        \path[->, red, thick] (B) edge[bend right] node[right, black] {$\widetilde{s}_{2B}$} (2);
        \draw[->,red](B) edge[loop below] node[below,black] {$1$} ();
    \end{tikzpicture}
\caption{A complete network $\mathcal{G}^{(B)}_{2}$ with two nodes and one bot $B$.}
\label{fig:two_nodes_one_bot}
\end{center}
\end{figure}
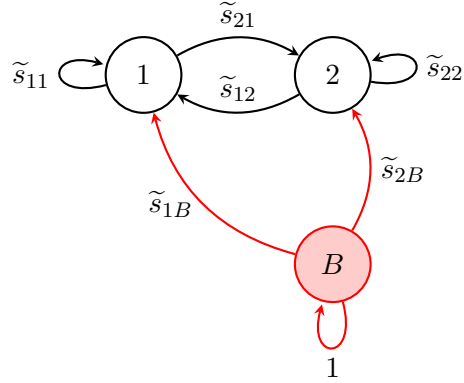
\begin{align}\label{eqn:fixed_pt_twonodes_onebot}
&P^{(B,*)} = \left(\begin{bmatrix}
    1 & 0\\
    0 & 1
\end{bmatrix}  - c_{2,M}M\begin{bmatrix}
    \widetilde{s}_{11} & \widetilde{s}_{12}\\
    \widetilde{s}_{21} & \widetilde{s}_{22}
\end{bmatrix} \right) ^{-1} \begin{bmatrix}
 (\widetilde{s}_{11} + \widetilde{s}_{12})c_{1,M} + \widetilde{s}_{1B}\eta_{B}\\
 (\widetilde{s}_{21} + \widetilde{s}_{22})c_{1,M} + \widetilde{s}_{2B}\eta_{B}
\end{bmatrix} \nonumber\\[2 em]
& = \frac{1}{1-c_{2,M}M(\widetilde{s}_{11}+\widetilde{s}_{22}) + c_{2,M}^{2}M^{2}(\widetilde{s}_{11}\widetilde{s}_{22}-\widetilde{s}_{21}\widetilde{s}_{12})}\begin{bmatrix}
1 -c_{2,M}M\widetilde{s}_{22} & c_{2,M}M\widetilde{s}_{12} \\
c_{2,M}M\widetilde{s}_{21} & 1-c_{2,M}M\widetilde{s}_{11}\end{bmatrix}\begin{bmatrix}
 (\widetilde{s}_{11} + \widetilde{s}_{12})c_{1,M} + \widetilde{s}_{1B}\eta_{B}\\
 (\widetilde{s}_{21} + \widetilde{s}_{22})c_{1,M}+ \widetilde{s}_{2B}\eta_{B}
\end{bmatrix} \nonumber\\
& = \begin{bmatrix}
   \frac{(1-c_{2,M}M\widetilde{s}_{22})(\widetilde{s}_{11}+\widetilde{s}_{12})c_{1,M} + (c_{2,M}M\widetilde{s}_{12})(\widetilde{s}_{21}+\widetilde{s}_{22})c_{1,M}+ \eta_{B}(\widetilde{s}_{1B}(1-c_{2,M}M\widetilde{s}_{22}) + c_{2,M}M\widetilde{s}_{12}\widetilde{s}_{2B})}{1-c_{2,M}M(\widetilde{s}_{11}+\widetilde{s}_{22}) + c_{2,M}^{2}M^{2}(\widetilde{s}_{11}\widetilde{s}_{22}-\widetilde{s}_{21}\widetilde{s}_{12})}\\\\
    \frac{(1-c_{2,M}M\widetilde{s}_{11})(\widetilde{s}_{21}+\widetilde{s}_{22})c_{1,M} + (c_{2,M}M\widetilde{s}_{21})(\widetilde{s}_{11}+\widetilde{s}_{12})c_{1,M}+ \eta_{B}(\widetilde{s}_{2B}(1-c_{2,M}M\widetilde{s}_{11}) + c_{2,M}M\widetilde{s}_{21}\widetilde{s}_{1B})}{1-c_{2,M}M(\widetilde{s}_{11}+\widetilde{s}_{22}) + c_{2,M}^{2}M^{2}(\widetilde{s}_{11}\widetilde{s}_{22}-\widetilde{s}_{21}\widetilde{s}_{12})}
\end{bmatrix}
 \end{align}
 We next study the influence of the bot on this fixed point by setting $(P^{(B,*)}_{i}-p^{(*)})>0$ for $i \in \{1,2\}$ and get a lower bound for the bot strength $\eta_{B}$. For ease of computation, we denote $D: = 1-c_{2,M}M(\widetilde{s}_{11}+\widetilde{s}_{22}) + c_{2,M}^{2}M^{2}(\widetilde{s}_{11}\widetilde{s}_{22}-\widetilde{s}_{21}\widetilde{s}_{12})$ and solve $(P^{(B,*)}_{1}-p^{(*)})>0$ to obtain the following lower bound on $\eta_{B}$ as follows:
 \begin{align}\label{eqn:eta_LBcalc}
\eta_{B} > \frac{c_{1,M}D-c_{1,M}(1-c_{2,M}M)[(1-c_{2,M}M\widetilde{s}_{22}(\widetilde{s}_{11} + \widetilde{s}_{12}) + c_{2,M}M\widetilde{s}_{12}(\widetilde{s}_{21}+\widetilde{s}_{22})]}{(1-c_{2,M}M)[\widetilde{s}_{1B}(1-c_{2,M}M\widetilde{s}_{22})+c_{2,M}M\widetilde{s}_{12}\widetilde{s}_{2B}]},    
 \end{align}
 which can be further simplified to:
 \begin{align}
   \eta_{B} > \frac{c_{1,M}}{(1-c_{2,M}M)}.  
 \end{align}
 
\noindent The same lower bound is obtained on solving $(P^{(B,*)}_{2}-p^{(*)})>0$.
\end{example}
\begin{example}\label{eg:general_network_same_bot_weight}(A general network with same bot weights on all the nodes)
The interaction matrix $\widetilde{S}_{N}^{(B)}$ for this network is given by:  

\begin{align}\label{eqn:S_one_symmetric_bot}
\widetilde{S}_{N}^{(B)} = \left[
\begin{array}{ccccc|c}
\widetilde{s}_{11} & \widetilde{s}_{12} & \widetilde{s}_{13} & \cdots & \widetilde{s}_{1N} & \alpha \\
\widetilde{s}_{21} & \widetilde{s}_{22} & \widetilde{s}_{23} & \cdots & \widetilde{s}_{2N} & \alpha\\
\widetilde{s}_{31} & \widetilde{s}_{32} & \widetilde{s}_{33} & \cdots & \widetilde{s}_{3N} & \alpha \\
\vdots & \vdots & \vdots & \ddots & \vdots & \vdots \\
\widetilde{s}_{N1} & \widetilde{s}_{N2} & \widetilde{s}_{N3} & \cdots & \widetilde{s}_{NN} & \alpha \\
\hline
0 & 0 & 0 & \cdots & 0 & 1 \\
\end{array}
\right]. 
\end{align}

We substitute $P^{(B,*)} := [P^{(B,*)}_{1},\cdots,P^{(B,*)}_{N}]^{T}$ and use \eqref{eqn:P_calc_bots} to obtain the following matrix equation:
\begin{align}\label{eqn:ds_homo_one_symmetric_bot}
P^{(B,*)} = (1-\alpha)c_{1,M}\mathbf{1}_{N} + c_{2,M}M\widetilde{S}_{N}P^{(B,*)} + \alpha\eta_{B}\mathbf{1}_{N}.   
\end{align}
Next, we subtract $\beta = \frac{(1-\alpha)c_{1,M} + \alpha\eta_{B}}{1-(1-\alpha)Mc_{2,M}}$ both sides in \eqref{eqn:ds_homo_one_symmetric_bot} as follows:
\begin{align}\label{eqn:fixed_pt_calc_one_symmetric_bot}
&P^{(B,*)} -\beta \mathbf{1}_{N} = \left[(1-\alpha)c_{1,M}+ \alpha\eta_{B} -  \frac{(1-\alpha)c_{1,M} + \alpha\eta_{B}}{1-(1-\alpha)Mc_{2,M}}\right]\mathbf{1}_{N} + c_{2,M}M\widetilde{S}_{N}P^{(B,*)}\nonumber\\[1.5 em]
&= \frac{-(1-\alpha)c_{2,M}M[(1-\alpha)c_{1,M}+\alpha\eta_{B}]\mathbf{1}_{N}}{1-(1-\alpha)Mc_{2,M}} = c_{2,M}M\widetilde{S}_{N}[P^{(B,*)}-\beta \mathbf{1}_{N}]
\end{align}
On substituting $\widetilde{P} = P^{(B,*)} - \beta \mathbf{1}_{N}$ in \eqref{eqn:fixed_pt_calc_one_symmetric_bot}, we obtain
\begin{align}\label{eqn:eig_equation}
\widetilde{S}_{N}\widetilde{P} = \frac{1}{c_{2,M}M}\widetilde{P},
\end{align}
Recall that the row sums in $\widetilde{S}_{N}$ are all given by $(1-\alpha)$. Therefore, the eigenvalues of $\widetilde{S}_{N}$ lie in the interval $(-1,1)$, which implies $\widetilde{P}=0 \iff P^{(B,*)} = \beta \mathbf{1}_{N}$. Furthermore, setting $\alpha =0$ in $P^{(B,*)}$, we obtain the fixed point $P^{(*)}$ for \eqref{eqn:P_eqn} with no bots, i.e.,  $P^{(*)} := c_{1,M}/(1-c_{2,M}M)\mathbf{1}_{N}$. We further get a lower bound on the strength of the bot $\eta_{B}$ for $P^{(B,*)}$ to be strictly greater than $P^{(*)}$:
\begin{align}\label{eqn:cond_eta_one_symmetric_bot}
\frac{(1-\alpha)c_{1,M} + \alpha\eta_{B}}{1-(1-\alpha)Mc_{2,M}} - \frac{c_{1,M}}{1- c_{2,M}M}> 0 \iff \eta_{B}> \frac{c_{1,M}}{1-c_{2,M}M}.
\end{align} 
\end{example}

\begin{example}\label{eg:complete_varying_bot}(A complete symmetric network with varying bot weights) The interaction matrix for this network is given by:  

\begin{align}\label{eqn:S_one_symmetric_bot}
\widetilde{S}_{N}^{(B)} = \left[
\begin{array}{ccccc|c}
\frac{(1-\alpha_{1})}{N}& \frac{(1-\alpha_{1})}{N} & \frac{(1-\alpha_{1})}{N} & \cdots & \frac{(1-\alpha_{1})}{N} & \alpha_{1} \\
\frac{(1-\alpha_{2})}{N} & \frac{(1-\alpha_{2})}{N} & \frac{(1-\alpha_{2})}{N} & \cdots & \frac{(1-\alpha_{2})}{N} & \alpha_{2}\\
\vdots & \vdots & \vdots & \ddots & \vdots & \vdots \\
\frac{(1-\alpha_{N})}{N} & \frac{(1-\alpha_{N})}{N} & \frac{(1-\alpha_{N})}{N} & \cdots & \frac{(1-\alpha_{N})}{N} & \alpha_{N} \\
\hline
0 & 0 & 0 & \cdots & 0 & 1 \\
\end{array}
\right]. 
\end{align}
Using \eqref{eqn:P_calc_bots}, we get the following form for the fixed point $P^{(B,*)}$ for this system:
\begin{align}\label{eqn:P_complete_varying_bots}
P^{(B,*)} = Nc_{1,M}\begin{bmatrix}
N-c_{2,M}M(1-\alpha_{1}) & -c_{2,M}M(1-\alpha_{1})& \cdots & -c_{2,M}M(1-\alpha_{1})\\
-c_{2,M}M(1-\alpha_{2}) & N-c_{2,M}M(1-\alpha_{2})& \cdots & -c_{2,M}M(1-\alpha_{2})\\
\vdots & \ddots & \vdots & \vdots \\
-c_{2,M}M(1-\alpha_{N}) & -c_{2,M}M(1-\alpha_{N})& \cdots & N-c_{2,M}M(1-\alpha_{N})
\end{bmatrix}^{-1}\begin{bmatrix}
1-\alpha_{1}\\
1-\alpha_{2}\\
\vdots\\
1-\alpha_{N}
\end{bmatrix}  
\end{align}
We can now use the Sherman-Morrison formula \cite[Section~0.8]{RAH-CRJ:12} to find the inverse of the matrix in \eqref{eqn:P_complete_varying_bots} as follows:

\vspace{0.2cm}

Set $u := \begin{bmatrix}
1-\alpha_{1}\\
1-\alpha_{2}\\
\vdots\\
1-\alpha_{N}
\end{bmatrix}; \quad \quad  v := (c_{2,M}M)^{1/2}\mathbf{1}_{N}$, then
\begin{align}\label{eqn:matrix_decom}
\begin{bmatrix}
N-c_{2,M}M(1-\alpha_{1}) & -c_{2,M}M(1-\alpha_{1})& \cdots & -c_{2,M}M(1-\alpha_{1})\\
-c_{2,M}M(1-\alpha_{2}) & N-c_{2,M}M(1-\alpha_{2})& \cdots & -c_{2,M}M(1-\alpha_{2})\\
\vdots & \ddots & \vdots & \vdots \\
-c_{2,M}M(1-\alpha_{N}) & -c_{2,M}M(1-\alpha_{N})& \cdots & N-c_{2,M}M(1-\alpha_{N})
\end{bmatrix} = NI_{N} + uv^{\mathbf{T}},
\end{align}
Hence, using Sherman-Morrison formula in \eqref{eqn:matrix_decom}, we obtain:
\begin{align}\label{eqn:Sherman-Morrison}
\begin{bmatrix}
N-c_{2,M}M(1-\alpha_{1}) & -c_{2,M}M(1-\alpha_{1})& \cdots & -c_{2,M}M(1-\alpha_{1})\\
-c_{2,M}M(1-\alpha_{2}) & N-c_{2,M}M(1-\alpha_{2})& \cdots & -c_{2,M}M(1-\alpha_{2})\\
\vdots & \ddots & \vdots & \vdots \\
-c_{2,M}M(1-\alpha_{N}) & -c_{2,M}M(1-\alpha_{N})& \cdots & N-c_{2,M}M(1-\alpha_{N})
\end{bmatrix}^{-1} = N\mathbf{1}_{N} - \frac{N\mathbf{1}_{N}uv^{\mathbf{T}}N\mathbf{1}_{N}}{1+v^{\mathbf{T}}N\mathbf{1}_{N}u}.
\end{align}

We now substitute the inverse obtained in \eqref{eqn:Sherman-Morrison} in the fixed point equation \eqref{eqn:P_complete_varying_bots}:

\begin{align}\label{eqn:fixed_pt_vector_complete_varying_bots}
P^{(B,*)} = Nc_{1,M}\left[ N\mathbf{1}_{N} - \frac{N\mathbf{1}_{N}uv^{\mathbf{T}}N\mathbf{1}_{N}}{1+v^{\mathbf{T}}N\mathbf{1}_{N}u}\right]\begin{bmatrix}
1-\alpha_{1}\\
\vdots\\
1-\alpha_{N}
\end{bmatrix},  
\end{align}
which gives
\begin{align}\label{eqn:fixed_pt_i_complete_varying_bots}
 P^{(B,*)}_{i} =  \frac{N(c_{1,M}(1-\alpha_{i})) + \eta_{B}[N\alpha_{i} + c_{2,M}M(\sum_{j\neq i}\alpha_{j} - (N-1)\alpha_{i})]}{N + c_{2,M}M(\sum_{k=1}^{N}\alpha_{k}-N)}.  
\end{align}
Similar to the previous examples, we can get a lower bound on $\eta_{B}$ by solving  $(P_{i}^{(B,*)}-p^{(*)})>0$ as follows:
\begin{align}\label{eqn:LB_complete_varying_bots_one}
&P^{(B,*)}_{i}- p^{(*)} = \frac{N(c_{1,M}(1-\alpha_{i})) + \eta_{B}[N\alpha_{i} + c_{2,M}M(\sum_{j\neq i}\alpha_{j} - (N-1)\alpha_{i})]}{N + c_{2,M}M(\sum_{k=1}^{N}\alpha_{k}-N)} - \frac{c_{1,M}}{1-c_{2,M}M}>0 \nonumber\\
&\iff \frac{N(1-c_{2,M}M)(c_{1,M}(1-\alpha_{i})) + \eta_{B}(1-c_{2}M)[N\alpha_{i} + c_{2,M}M(\sum_{j\neq i}\alpha_{j} - (N-1)\alpha_{i})]}{(N + c_{2,M}M(\sum_{k=1}^{N}\alpha_{k}-N))(1-c_{2,M}M)} \nonumber\\
& -\frac{c_{1,M}(N-c_{2,M}M(N-\sum_{k=1}^{N}\alpha_{k}))}{(N + c_{2,M}M(\sum_{k=1}^{N}\alpha_{k}-N))(1-c_{2,M}M)}>0
\end{align}
Note that the denominator in \eqref{eqn:LB_complete_varying_bots_one} is always positive, and therefore we can write:
\begin{align}\label{eqn:LB_complete_varying_bots_two}
\eta_{B}> \frac{c_{1,M}(N-c_{2,M}M(N-\sum_{k=1}^{N}\alpha_{k}))-Nc_{1,M}(1-\alpha_{i})(1-c_{2,M}M)}{(1-c_{2}M)[N\alpha_{i} + c_{2,M}M(\sum_{j\neq i}\alpha_{j} - (N-1)\alpha_{i})]} = \frac{c_{1,M}}{1-c_{2,M}M}  
\end{align}
\end{example}

To understand the influence of bots on the network, we now compute $P^{(\mathcal{B},*)}$ for a special class of networks, known as \textit{The Hub-and-Spoke model}, which consists of a central node called the \emph{hub} which has connections to all the other nodes (\emph{spokes}). The spokes are not connected to each other. Such networks have variety of applications, primarily in healthcare \cite{EF:17, V:23} and air transport \cite{SSD:99,OK:98}. In the next two examples, the interacting networks consist of one central node (called the ``Hub'') that affects every other node (``Spokes'') in the network (including itself). We analyze two scenarios of bot attachment for this network: \textit{(i)} when the bot attaches only to the hub. \textit{(ii)} when the bot attaches to every node in the network. 
\begin{example}\label{eg:HS_bot_tohub}(hub-and-spoke model with bot influencing only the hub as shown in Fig.~\ref{fig:hub_spoke_bot_tohub}) The interaction matrix for this network is given by: 
\begin{align}\label{eqn:S_one_symmetric_bot}
\widetilde{S}_{N}^{(B)} = \left[
\begin{array}{ccccc|c}
\widetilde{s}_{11} & 0 & 0 & \cdots & 0 & \widetilde{s}_{1B} \\
1 & 0 & 0 & \cdots & 0 & 0\\
\vdots & \vdots & \vdots & \ddots & \vdots & \vdots \\
1 & 0 & 0 & \cdots & 0 & 0 \\
\hline
0 & 0 & 0 & \cdots & 0 & 1 \\
\end{array}
\right]. 
\end{align}
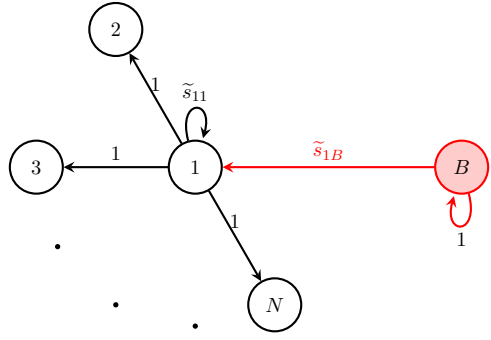
\begin{figure}[h!]
\begin{center}
\begin{tikzpicture}[
    transform shape,
    ->,
    >=stealth,
    node distance=2.0cm,
    thick, scale=0.7
]
% Define central node (vertex 1)
    \node[circle, draw, minimum size=1cm] (v1) {$1$};
    % Draw self-loop at vertex 1
    \draw[->] (v1) edge [loop above] node[above] {\(\widetilde{s}_{11}\)} (v1);

    % Define positions for outer vertices in a circular pattern
    \def\radius{3cm}  % Radius of circular arrangement
    \def\angleStep{60} % Angle step between vertices

    % Place first vertex (2)
    \node[circle, draw, minimum size=1cm] 
          (v2) at ({\radius*cos(\angleStep*2)}, {\radius*sin(\angleStep*2)}) {$2$};
    \draw[->] (v1) -- (v2) node[midway, above] {$1$};

    % Place second vertex (3)
    \node[circle, draw, minimum size=1cm] 
          (v3) at ({\radius*cos(\angleStep*3)}, {\radius*sin(\angleStep*3)}) {$3$};
    \draw[->] (v1) -- (v3) node[midway, above] {$1$};

    % Place dots in a circular pattern
    \node[circle, fill, inner sep=1pt] at ({\radius*cos(\angleStep*3.5)}, {\radius*sin(\angleStep*3.5)}) {};

    \node[circle, fill, inner sep=1pt] at ({\radius*cos(\angleStep*4.0)}, {\radius*sin(\angleStep*4.0)}){};

     \node[circle, fill, inner sep=1pt] at ({\radius*cos(\angleStep*4.5)}, {\radius*sin(\angleStep*4.5)}){};

    % Place last vertex (N)
    \node[circle, draw, minimum size=1cm] 
          (vN) at ({\radius*cos(\angleStep*5)}, {\radius*sin(\angleStep*5)}) {$N$};
    \draw[->] (v1) -- (vN) node[midway, above] {$1$};

    % Define special red vertex B
    \node[circle, draw=red,fill=red!20, minimum size=1cm, right=4cm of v1] (B) {$B$};

    % Draw red edge from B to 1 with weight s_{1B}
    \draw[->, red] (B) -- (v1) node[midway, above] {\(\widetilde{s}_{1B}\)};
    \draw[->,red] (B) edge[loop below] node[below,text=black] {$1$} ();    
\end{tikzpicture}
 \caption{An $N$-node hub-and-spoke network in which node $1$ is the hub and remaining $N-1$ individuals are spokes. There is a single bot, $B$, present in the network that influences the opinion of the hub.}
\label{fig:hub_spoke_bot_tohub}
\end{center}
\end{figure}
Due to symmetry of the network, the fixed point for \eqref{eqn:P_calc_bots} is of the form $[\beta_{H,M},\beta_{S,M},\cdots, \beta_{S,M}]^{T}$, where $\beta_{H,M}$ and $\beta_{S,M}$ are the fixed points for the hub-and-spokes respectively. Using \eqref{eqn:P_calc_bots}, we obtain the expression for $\beta_{H,M}$:
\begin{align}\label{eqn:fixed_pt_hub_bot_tohub_M}
\beta_{H,M} =\widetilde{s}_{11}c_{1,M} + \widetilde{s}_{11}c_{2,M}M\beta_{H,M} + \widetilde{s}_{1B}\eta_{B} \iff \beta_{H,M} = \frac{\widetilde{s}_{11}c_{1,M} + \widetilde{s}_{1B}\eta_{B}}{1-\widetilde{s}_{11}c_{2,M}M}
\end{align}
We next obtain the expression for $\beta_{S,M}$, using~\eqref{eqn:P_calc_bots} and~\eqref{eqn:fixed_pt_hub_bot_tohub_M} as follows:
\begin{align}\label{eqn:fixed_pt_spoke_bot_tohub}
\beta_{S,M} = c_{2,M}M\beta_{H,M} + c_{1,M} \iff \beta_{S,M} = \frac{c_{1,M} + \widetilde{s}_{1,B}\eta_{B}c_{2,M}M}{1-\widetilde{s}_{11}c_{2,M}M}
\end{align}
\end{example}
\begin{example}\label{eg:HS_bot_toall}(hub-and-spoke model with bot influencing all the nodes as shown in Fig.~\ref{fig:hub_spoke_bot_toall}) 
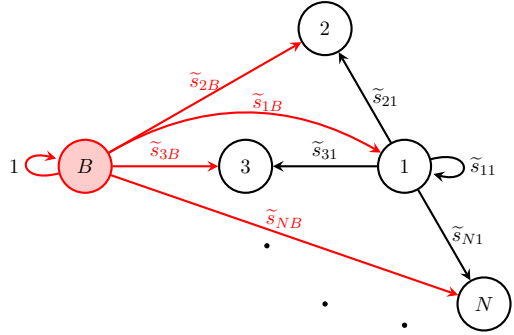
\begin{figure}[h!]
\begin{center}
\begin{tikzpicture}[
    transform shape,
    ->,
    >=stealth,
    node distance=2.0cm,
    thick, scale=0.7
]
% Define central node (vertex 1)
    \node[circle, draw, minimum size=1cm] (v1) {$1$};
    % Draw self-loop at vertex 1
    \draw[->] (v1) edge [loop right] node[above,right] {\(\widetilde{s}_{11}\)} (v1);

    % Define positions for outer vertices in a circular pattern
    \def\radius{3cm}  % Radius of circular arrangement
    \def\angleStep{60} % Angle step between vertices

    % Place first vertex (2)
    \node[circle, draw, minimum size=1cm] 
          (v2) at ({\radius*cos(\angleStep*2)}, {\radius*sin(\angleStep*2)}) {$2$};
    \draw[->] (v1) -- (v2) node[midway, above,right] {$\widetilde{s}_{21}$};

    % Place second vertex (3)
    \node[circle, draw, minimum size=1cm] 
          (v3) at ({\radius*cos(\angleStep*3)}, {\radius*sin(\angleStep*3)}) {$3$};
    \draw[->] (v1) -- (v3) node[midway, above] {$\widetilde{s}_{31}$};

    % Place dots in a circular pattern
    \node[circle, fill, inner sep=1pt] at ({\radius*cos(\angleStep*3.5)}, {\radius*sin(\angleStep*3.5)}) {};

    \node[circle, fill, inner sep=1pt] at ({\radius*cos(\angleStep*4.0)}, {\radius*sin(\angleStep*4.0)}){};

     \node[circle, fill, inner sep=1pt] at ({\radius*cos(\angleStep*4.5)}, {\radius*sin(\angleStep*4.5)}){};

    % Place last vertex (N)
    \node[circle, draw, minimum size=1cm] 
          (vN) at ({\radius*cos(\angleStep*5)}, {\radius*sin(\angleStep*5)}) {$N$};
    \draw[->] (v1) -- (vN) node[midway,above,right] {$\widetilde{s}_{N1}$};

    % Define special red vertex B
    \node[circle, draw=red,fill=red!20, minimum size=1cm, left=2cm of v3] (B) {$B$};
    \draw[->,red] 
(B) edge[loop left] node[left,text=black] {$1$} ();

    % Draw red edge from B to 1,2,3 with respective weights
    \draw[->, red] (B) to[bend left] node[above=0.2,right] {\(\widetilde{s}_{1B}\)} (v1);
    \draw[->, red] (B) to  node[midway, above] {\(\widetilde{s}_{2B}\)} (v2);
    \draw[->, red] (B) to node[midway, above] {\(\widetilde{s}_{3B}\)} (v3);
    \draw[->, red] (B) to node[midway, above] {\(\widetilde{s}_{NB}\)} (vN);
    
\end{tikzpicture}
\caption{hub-and-spoke network with $N$ flexible nodes such that $1$ is the hub and remaining $N-1$ individuals are spokes. There is a single bot $B$ present in the network, influencing all the nodes in the network.}
\label{fig:hub_spoke_bot_toall}
\end{center}
\end{figure}
The interaction matrix for this network is given by:
\begin{align}\label{eqn:S_hub_spoke_bot_toall}
\widetilde{S}^{(B)}_{N} = \left[
\begin{array}{ccccc|c}
\widetilde{s}_{11} & 0 & 0 & \cdots & 0 & \widetilde{s}_{1B} \\
\widetilde{s}_{21} & 0 & 0 & \cdots & 0 & \widetilde{s}_{2B}\\
\widetilde{s}_{31} & 0 & 0 & \cdots & 0 & \widetilde{s}_{3B} \\
\vdots & \vdots & \vdots & \ddots & \vdots & \vdots \\
\widetilde{s}_{N1} & 0 & 0 & \cdots & 0 & \widetilde{s}_{NB} \\
\hline
0 & 0 & 0 & \cdots & 0 & 1 \\
\end{array}
\right].    
\end{align}
Similar to the previous examples, we use \eqref{eqn:P_calc_bots} to compute the fixed point $P^{(B,*)}$ as follows:
\begin{align}\label{eqn:fixed_pt_hub_spoke_bot_toall}
P^{(B,*)} = \begin{bmatrix}
1-c_{2,M}M\widetilde{s}_{11} & 0 & \cdots & 0\\
-c_{2,M}M\widetilde{s}_{21} & 1 & \cdots & 0\\
\vdots & \vdots & \ddots & \vdots \\
-c_{2,M}M\widetilde{s}_{N1} & 0 & \cdots & 1
\end{bmatrix}^{-1}\begin{bmatrix}
c_{1,M}\widetilde{s}_{11} + \widetilde{s}_{1B}\eta_{B}\\
c_{1,M}\widetilde{s}_{21} + \widetilde{s}_{2B}\eta_{B}\\
\vdots\\
c_{1,M}\widetilde{s}_{N1} + \widetilde{s}_{NB}\eta_{B}
\end{bmatrix} = \begin{bmatrix}
\frac{c_{1,M}\widetilde{s}_{11} + \widetilde{s}_{1B}\eta_{B}}{1-c_{2,M}M\widetilde{s}_{11}}\\\\
\frac{c_{1,M}\widetilde{s}_{21} + \eta_{B}[\widetilde{s}_{2B} + c_{2,M}M(\widetilde{s}_{21}\widetilde{s}_{1B} -\widetilde{s}_{11}\widetilde{s}_{2B})]}{1-c_{2,M}M\widetilde{s}_{11}}\\\\
\vdots\\\\
\frac{c_{1,M}\widetilde{s}_{N1} + \eta_{B}[\widetilde{s}_{NB} + c_{2,M}M(\widetilde{s}_{N1}\widetilde{s}_{1B} -\widetilde{s}_{11}\widetilde{s}_{NB})]}{1-c_{2,M}M\widetilde{s}_{11}}
\end{bmatrix}
\end{align}  
\end{example}
Note that the fixed point of hub is the same in the two preceding examples. To compare the fixed points of spokes, we denote $i$th entry of \eqref{eqn:fixed_pt_hub_spoke_bot_toall} by $\beta_{S,M}^{(i)}$ and solve for $(\beta_{S,M} - \beta_{S,M}^{(i)})>0$, where $\beta_{S,M}$ is the fixed point for spoke obtained in \eqref{eqn:fixed_pt_spoke_bot_tohub} :
\begin{align}\label{eqn:fixed_pt_compare_hubspoke}
(\beta_{S,M} - \beta_{S,M}^{(i)}) 
> 0 \iff \widetilde{s}_{iB}[c_{1,M} + \eta_{B}(c_{2,M}M-1)]>0 \iff \eta_{B} < \frac{c_{1,M}}{(1-c_{2,M}M)} = p^{(*)}.
\end{align}
A key observation in the previous two examples is that the fixed point for the hub is same in  examples~\ref{eg:HS_bot_tohub} and~\ref{eg:HS_bot_toall}, while it can be easily checked that the fixed point of each spoke in Example~\ref{eg:HS_bot_tohub} is greater than or equal to its fixed point in example~\ref{eg:HS_bot_toall}. From~\eqref{eqn:fixed_pt_compare_hubspoke}, $(\beta_{S,M} - \beta_{S,M}^{(i)}) > 0$ if and only if $\eta_B < p^{(*)}$, and 
$(\beta_{S,M} - \beta_{S,M}^{(i)}) < 0$ if and only if $\eta_B > p^{(*)}$. Therefore, a comparison 
between the two bot-attachment strategies (hub-only vs.\ all-nodes) depends critically on whether 
the bot strength exceeds the network fixed point:

\begin{itemize}
    \item If $\eta_B < p^{(*)}$ (bot biased toward opinion~$0$): attaching the bot to all nodes results in 
    a \emph{lower} fixed-point belief for spokes than attaching it only to the hub. 
    This is because the direct negative reinforcement from the bot reaches each spoke without 
    attenuation through the hub.
    \item If $\eta_B > p^{(*)}$ (bot biased toward opinion~$1$): attaching the bot to all nodes results in 
    a \emph{higher} fixed-point belief for spokes than attaching it only to the hub. 
    The direct positive reinforcement is again stronger than the attenuated effect via the hub.
\end{itemize}
 Furthermore, we leave it to the readers to check that $(P_{i}^{(B,*)}-p^{(*)})>0 \iff \eta_{B}>c_{1,M}/(1-c_{2,M}M)$ in Examples~\ref{eg:HS_bot_tohub} and~\ref{eg:HS_bot_toall}. A key result obtained in all the above examples is that setting $(P^{(B,*)}-p^{(*)})>0$ gives the same lower bound for the strength of the bot attached. Indeed, we now show in the following theorem that in the presence of a single bot $B$, this lower bound is obtained for $\eta_{B}$ irrespective of the network structure.

\begin{theorem}\label{thm:lower_bound_one_bot}
Given an interacting network of $N$ individuals $\mathcal{G}_{N}^{(\mathcal{B})}$ equipped with the interaction matrix~\eqref{eqn:S_bots}, $c_{2,M}\in (0,1/M)$ in \eqref{eqn:ratio_X_bots} and memory sets $\mathcal{M}^{(t)}_{ji} = \{t-k_{ji}^{(1)},\cdots, t-k_{ji}^{(M)}\}$ for all $i,j \in \{1,\cdots,N\}$ and a bot set $\mathcal{B}= \{B\}$, the following holds:
\begin{align}\label{eqn:eta_lowerbound}(P^{(B,*)}-P^{(*)})\succ \mathbf{0}_{N} \iff \eta_{B} > p^{(*)}.
\end{align} 
Furthermore,
\begin{align}\label{eqn:eta_equals_fixedpt}(P^{(B,*)}-P^{(*)})= \mathbf{0}_{N} \iff \eta_{B} = p^{(*)}.
\end{align} 
\end{theorem}
\textit{Proof:} For the given network, we can rewrite \eqref{eqn:P_calc_bots} as follows: 
\begin{align}\label{eqn:P_calc_bots_one}
P_{i}^{(B)}(t) = c_{1,M}\sum_{j=1}^{N}\widetilde{s}_{ij} + \widetilde{s}_{iB}\eta_{B}+ c_{2,M}\sum_{j=1}^{N}\sum_{
n\in \mathcal{M}_{ji}^{(t)}}\hspace{-0.2cm}P_{j}^{(B)}(t-n).   
\end{align}
Corresponding to $\mathcal{G}_{N}^{(B)}$, we consider a homogeneous network $\mathcal{G}_{N}$ without bots and equipped with the memory sets $\mathcal{M}^{(t)}_{ji} = \{t-k_{ji}^{(1)},\cdots, t-k_{ji}^{(M)}\}$ for all $i,j \in \{1,\cdots,N\}$. We further assume the following two conditions on its interaction matrix $S = [s_{ij}]_{i,j=1}^{N}$:
\begin{itemize}
\item[(1)] ${s}_{ij}=0 \iff \widetilde{s}_{ij}=0$,
\item[(2)] $\widetilde{s}_{ij} \leq s_{ij}$ for all $i,j \in \{1,\cdots,N\}$.
\end{itemize}

\noindent Note that the above two conditions ensure that $\widetilde{s}_{i,B}= \sum_{j=1}^{N}(s_{ij}-\widetilde{s}_{ij})$ as well as the same graph structure for $\mathcal{G}_{N}^{(B)}$ and $\mathcal{G}_{N}$. Setting $Y_{i}^{(B)}(t) := (P_{i}^{(B)}(t) - P_{i}(t))$ and using \eqref{eqn:P_calc} and \eqref{eqn:P_calc_bots_one}, we obtain the following:
\begin{align}\label{eqn:Y_calc}
P_{i}^{(B)}(t) - P_{i}(t) = c_{1,M}(\sum_{j=1}^{N}\widetilde{s}_{ij}-1) + \widetilde{s}_{iB}\eta_{B}+ c_{2,M} \sum_{j=1}^{N}\big[\hspace{-0.25cm}\sum_{n \in \{k_{ji}^{(1)},\cdots,k_{ji}^{(M)}\}}\hspace{-0.8cm}\widetilde{s}_{ij}(P_{j}^{(B)}(t-n) - P_{j}(t-n))\big].
\end{align}
We now denote $Y_{i}^{(B,*)}= \lim\limits_{t \to \infty}(P_{i}^{(B)}(t)-P_{i}(t))$, and take $t \to \infty$ in \eqref{eqn:Y_calc}:
\begin{align}\label{eqn:Y_limit_calc_one}
Y_{i}^{(B,*)} \hspace{-0.1cm}= \hspace{-0.1cm}\widetilde{s}_{iB}(\eta_{B}-c_{1,M}) + c_{2,M}M\sum_{j=1}^{N}\big[\widetilde{s}_{ij}P_{j}^{(B,*)}-s_{ij}p^{(*)}\big].  
\end{align}
Adding and subtracting $c_{2,M}M(\sum_{j=1}^{N}\widetilde{s}_{ij}p^{(*)})$ to \eqref{eqn:Y_limit_calc_one}, we obtain:
\begin{align}\label{eqn:Y_limit_calc_two}
Y_{i}^{(B,*)} &= \widetilde{s}_{iB}(\eta_{B}- c_{1,M}) + c_{2,M}M \sum_{j=1}^{N}\big[\widetilde{s}_{ij}Y_{j}^{(B,*)} +(\widetilde{s}_{ij}-s_{ij})p^{(*)}\big]\nonumber\\
 &= \widetilde{s}_{iB}\bigg(\eta_{B}-c_{1,M}-\frac{c_{1,M}c_{2,M}M}{1-Mc_{2,M}}\bigg) + c_{2,M}M\sum_{j=1}^{N}\widetilde{s}_{ij}Y_{j}^{(B,*)},
\end{align}
where we have substituted $p^{(*)} = c_{1,M}/(1-Mc_{2,M})$ for all $j \in \{1,\cdots,N\}$, and $\sum_{j=1}^{N}\widetilde{s}_{iB} = \sum_{j=1}^{N}(s_{ij}-\widetilde{s}_{ij})$. To write the matrix form for \eqref{eqn:Y_limit_calc_two}, we let $Y^{(B,*)} := [Y_{1}^{(B,*)},\cdots,Y_{N}^{(B,*)}]^{\mathbf{T}}$ to obtain:
\begin{align}\label{eqn:Y_matrix}
Y^{(B,*)} = D_{N}^{(B)}\bigg(\frac{\eta_{B}(1-Mc_{2,M})-c_{1,M}}{1-Mc_{2,M}}\bigg)\mathbf{1}_{N} + c_{2,M}M\widetilde{S}_{N}Y^{(B,*)},     
\end{align}
where $D_{N}^{(B)}$ is a diagonal matrix with $[D]_{ii} = \widetilde{s}_{iB}$. Solving for $Y^{(B,*)}$ in \eqref{eqn:Y_matrix}, we get
\begin{align}\label{eqn:Y_formula}
Y^{(B,*)}= \frac{(\eta_{B}(1-Mc_{2,M})-c_{1,M})(I_{N\times N}-c_{2,M}M\widetilde{S}_{N})^{-1}D_{N}^{(B)}\mathbf{1}_{N}}{1-Mc_{2,M}}.
\end{align}
Note that in \eqref{eqn:Y_formula} the matrix $c_{2,M}M\widetilde{S}_{N}$ has spectral radius strictly less than one, and hence $(I_{N\times N}-c_{2,M}M\widetilde{S}_{N})$ is an invertible matrix. Furthermore, writing the corresponding Neumann series expansion for $(I_{N\times N}-c_{2,M}M\widetilde{S}_{N})^{-1}$, we obtain that all the entries of this inverse are non-negative and at least one entry in each row of $(I_{N\times N}-c_{2,M}M\widetilde{S}_{N})^{-1}$ is positive. Therefore, $(I_{N\times N}-c_{2,M}M\widetilde{S}_{N})^{-1}D_{N}^{(B)}\mathbf{1}_{N} =(I_{N\times N}-c_{2,M}M\widetilde{S}_{N})^{-1}(\widetilde{s}_{1,B},\cdots,\widetilde{s}_{N,B})^{\mathbf{T}}\succ \mathbf{0}_{N}$ which gives (using \eqref{eqn:Y_formula}) $Y^{(B,*)}\succ \mathbf{0}_{N} \iff \eta_{B}> p^{(*)}$ and $Y^{(B,*)}=\mathbf{0}_{N} \iff \eta_{B}= p^{(*)}$.
$\null\nobreak\hfill\ensuremath{\square}$

For the one-bot setting, Theorem~\ref{thm:lower_bound_one_bot} indicates that in order to perturb the a homogeneous system (without bots) towards a bot favored opinion, the added bot should have a strength greater than $p^{(*)}$. Furthermore, the fixed point does not shift when bot strength exactly equals the entry of the fixed point vector (i.e., $\eta_{B} = p^{(*)}$). In the latter case, the bot behaves like a typical individual which has reached the fixed point in a homogeneous network and has no effect on the asymptotics of the network. In the next example, we analyze the dynamics of a two-node network equipped with two bots of varying strengths. 
\begin{example}\label{eg:two_nodes_two_bots}(A general network on two nodes equipped with two bots).
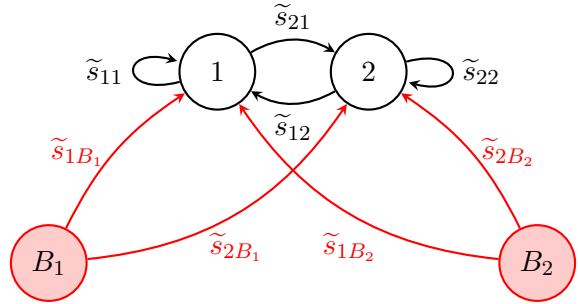
\begin{figure}[h!]
\begin{center}
\begin{tikzpicture}[
    transform shape,
    ->,
    >=stealth,
    node distance=2.0cm,
    thick
]

% Nodes
\node[circle, draw, minimum size=1cm] (1) {$1$};
\node[circle, draw, minimum size=1cm, right of=1] (2) {$2$};

\node[circle, draw, fill=red!20, draw=red, minimum size=1cm, below left=1.8cm and 1.5cm of 1] (B1) {$B_1$};
\node[circle, draw, fill=red!20, draw=red, minimum size=1cm, below right=1.8cm and 1.5cm of 2] (B2) {$B_2$};

% Black edges (typical nodes)
\path (1) edge[bend left] node[above] {$\widetilde{s}_{21}$} (2);
\path (2) edge[bend left] node[below] {$\widetilde{s}_{12}$} (1);

\path (1) edge[loop left] node {$\widetilde{s}_{11}$} (1);
\path (2) edge[loop right] node {$\widetilde{s}_{22}$} (2);

% Red edges from bots
\path[red]
(B1) edge[bend left=15] node[left] {$\widetilde{s}_{1B_1}$} (1)
(B1) edge[bend right=25] node[below=3pt] {$\widetilde{s}_{2B_1}$} (2);

\path[red]
(B2) edge[bend right=15] node[right] {$\widetilde{s}_{2B_2}$} (2)
(B2) edge[bend left=25] node[below=3pt] {$\widetilde{s}_{1B_2}$} (1);

\end{tikzpicture}

\caption{A two-node directed network $\mathcal{G}^{(B)}_{2}$ with two bots $B_1$ and $B_2$.}
\label{fig:two_nodes_two_bots}
\end{center}
\end{figure}
\begin{align}
P^{(\mathcal{B},*)}
= \left(
\begin{bmatrix}
1 & 0 \\
0 & 1
\end{bmatrix}
-
c_{2,M} M
\begin{bmatrix}
\widetilde{s}_{11} & \widetilde{s}_{12} \\
\widetilde{s}_{21} & \widetilde{s}_{22}
\end{bmatrix}
\right)^{-1} \begin{bmatrix}
(\widetilde{s}_{11}+\widetilde{s}_{12})c_{1,M}
+ \widetilde{s}_{1,B_1}\eta_{B_1}
+ \widetilde{s}_{1,B_2}\eta_{B_2}
\\
(\widetilde{s}_{21}+\widetilde{s}_{22})c_{1,M}
+ \widetilde{s}_{2,B_1}\eta_{B_1}
+ \widetilde{s}_{2,B_2}\eta_{B_2}
\end{bmatrix},
\end{align}
 To obtain a correlation between strength of these two bots, we set $(P^{(\mathcal{B},*)}-P^{(*)})=\mathbf{0}_{2}$. This happens when the effect of one bot counters the other. Due to space limitations, we omit the calculations and only write the final expression obtained by setting $(P_{1}^{(\mathcal{B},*)}-P_{1}^{(*)})=0$:
\begin{align}\label{eqn:two_bots_compare_one}
C^{(1)}_{B_{1}}\Big[\eta_{B_{1}}-p^{(*)}\Big] + C^{(1)}_{B_{2}}\Big[\eta_{B_{2}}-p^{(*)}\Big]  = 0 
\end{align}
%\frac{c_{1,M}c_{2,M}^{2}M^{2}\widetilde{s}_{12}\widetilde{s}_{22}}{1-c_{2,M}M}

where, $C^{(1)}_{B_{1}} = \widetilde{s}_{1B_{1}}(1-c_{2,M}M\widetilde{s}_{22}) + c_{2,M}M\widetilde{s}_{12}\widetilde{s}_{2B_{1}}$ and $C^{(1)}_{B_{2}} = \widetilde{s}_{1B_{2}}(1-c_{2,M}M\widetilde{s}_{22}) + c_{2,M}M\widetilde{s}_{12}\widetilde{s}_{2B_{2}}$. Similarly, setting $(P_{2}^{(B,*)}-P_{2}^{(*)})=0$, we obtain: 
\begin{align}\label{eqn:two_bots_compare_two}
C^{(2)}_{B_1}
\left[
\eta_{B_1}
-
p^{(*)}
\right]
+
C^{(2)}_{B_2}
\left[
\eta_{B_2}
-
p^{(*)}
\right]
= 0,
\end{align}
where, $
C^{(2)}_{B_1}
=
\widetilde{s}_{2B_1}(1-c_{2,M}M\widetilde{s}_{11})
+
c_{2,M}M\,\widetilde{s}_{21}\widetilde{s}_{1B_1}$ and $C^{(2)}_{B_2}
=
\widetilde{s}_{2B_2}(1-c_{2,M}M\widetilde{s}_{11})
+
c_{2,M}M\,\widetilde{s}_{21}\widetilde{s}_{1B_2}$. Note that, setting $\eta_{B_{1}}=\eta_{B_{2}}= p^{(*)}$ satisfy both \eqref{eqn:two_bots_compare_one} and \eqref{eqn:two_bots_compare_two}, i.e., $P^{(\mathcal{B},*)}= P^{(*)}$. Similar to the one bot scenario, both the bots in this case behave like typical individuals who have reached fixed point and do not perturb the asymptotics of the network. The interesting cases to study for \eqref{eqn:two_bots_compare_one} and \eqref{eqn:two_bots_compare_two} being satisfied are the ones with $\eta_{B_{1}}\neq p^{(*)}$ and $\eta_{B_{2}}\neq p^{(*)}$. In these cases, one bot nullifies the effect of another and therefore $P^{(\mathcal{B},*)} = P^{(*)}$. To this end, we write a combined matrix equation for \eqref{eqn:two_bots_compare_one} and \eqref{eqn:two_bots_compare_two}:
\begin{align}\label{eqn:two_bots_compare_matrix}
\begin{bmatrix}
C_{B_{1}}^{(1)} && C_{B_{2}}^{(1)} \\\\
C_{B_{1}}^{(2)} & &C_{B_{2}}^{(2)}
\end{bmatrix}\begin{bmatrix}
\eta_{B_{1}}-p^{(*)}\\\\
\eta_{B_{2}}-p^{(*)}
\end{bmatrix}  = \begin{bmatrix}
    0 \\\\
    0
\end{bmatrix}.
\end{align}

One can easily show that the determinant of the $2\times 2$ matrix in \eqref{eqn:two_bots_compare_matrix} is given by:
\begin{align}\label{eqn:det_C}\det &\begin{bmatrix}
C_{B_{1}}^{(1)} && C_{B_{2}}^{(1)} \\\\
C_{B_{1}}^{(2)} & &C_{B_{2}}^{(2)}
\end{bmatrix} =[\widetilde{s}_{1B_{1}}\widetilde{s}_{2B_{2}}-\widetilde{s}_{1B_{2}}\widetilde{s}_{2B_{1}}][1-c_{2,M}M(\widetilde{s}_{11}+\widetilde{s}_{22})+c_{2,M}^{2}M^{2}[\widetilde{s}_{11}\widetilde{s}_{22}+\widetilde{s}_{12}\widetilde{s}_{21}]].
\end{align}
A non-trivial solution for \eqref{eqn:two_bots_compare_matrix} (i.e., $\eta_{B_{1}}\neq p^{(*)}$ and $\eta_{B_{2}}\neq p^{(*)}$) exists if and only if the determinant in \eqref{eqn:det_C} is zero, which occurs if and only if  $\widetilde{s}_{1B_{1}}\widetilde{s}_{2B_{2}} = \widetilde{s}_{1B_{2}}\widetilde{s}_{2B_{1}}  $. Assuming that each bot has at least one outgoing edge, $\widetilde{s}_{1B_{1}}\widetilde{s}_{2B_{2}} = \widetilde{s}_{1B_{2}}\widetilde{s}_{2B_{1}}$ in the following three cases:
\begin{itemize}
 \item[] \underline{Case 1}: $\widetilde{s}_{1B_{1}} =\widetilde{s}_{1B_{2}} =0$, $\widetilde{s}_{2B_{1}}\neq 0$ and $\widetilde{s}_{2B_{2}} \neq 0$ i.e., both the bots do not have an outgoing edge to node $1$, but are connected to node $2$. For this case, the matrix equation \eqref{eqn:two_bots_compare_matrix} gives:
 \begin{align}\label{eqn:eta_ratio_degenerate_one}
\frac{(\eta_{B_{1}}-p^{(*)})}{(\eta_{B_{2}}-p^{(*)})} = \frac{-\widetilde{s}_{2B_{2}}}{\widetilde{s}_{2B_{1}}}.  
 \end{align}

\item[]\underline{Case 2}: $\widetilde{s}_{2B_{1}}=\widetilde{s}_{2B_{2}}=0$, $\widetilde{s}_{1B_{1}}\neq 0$ and $\widetilde{s}_{1B_{2}}\neq 0$  i.e, both the bots do not have an outgoing edge to node $2$, but are connected to node $1$. Similar to the previous case, we get the following solution for \eqref{eqn:two_bots_compare_matrix} here:
\begin{align}\label{eqn:eta_ratio_degenerate_two}
\frac{(\eta_{B_{1}}-p^{(*)})}{(\eta_{B_{2}}-p^{(*)})} = \frac{-\widetilde{s}_{1B_{2}}}{\widetilde{s}_{1B_{1}}}.     
\end{align}

\item[]\underline{Case 3}: $\widetilde{s}_{1B_{1}}$, $\widetilde{s}_{1B_{2}}$, $\widetilde{s}_{2B_{1}}$ and $\widetilde{s}_{2B_{2}}$ are all positive. Setting $\lambda = \widetilde{s}_{1B_{1}}/\widetilde{s}_{1B_{2}} = \widetilde{s}_{2B_{1}}/\widetilde{s}_{2B_{2}}$, we obtain $C_{B_{1}}^{(1)} = \lambda C_{B_{2}}^{(1)}$ and $C_{B_{1}}^{(2)} = \lambda C_{B_{2}}^{(2)}$. Therefore, solving \eqref{eqn:two_bots_compare_matrix}, we obtain:
\begin{align}\label{eqn:eta_ratio}
 \frac{(\eta_{B_{1}}-p^{(*)})}{(\eta_{B_{2}}-p^{(*)})} = \frac{-1}{\lambda}.
\end{align}
\end{itemize} 

Note that the right-hand side for \eqref{eqn:eta_ratio_degenerate_one},\eqref{eqn:eta_ratio_degenerate_two} and \eqref{eqn:eta_ratio} are all negative which implies that for the bots to nullify the effect of one another and for the system to achieve $P^{(\mathcal{B},*)}=P^{(*)}$, the strength of one bot should be strictly less than $p^{(*)}$ while that of the other must be strictly greater than $p^{(*)}$.
\end{example} 
We can indeed extend \eqref{eqn:eta_ratio_degenerate_one}, \eqref{eqn:eta_ratio_degenerate_two} and \eqref{eqn:eta_ratio} to obtain non-trivial solution for bot strengths for a $N$-node network $\mathcal{G}_{N}^{(\mathcal{B})}$, where $\mathcal{B}$ is an arbitrary bot set.
\begin{theorem}\label{thm:multi_bots}
Given an interacting network of individuals $\mathcal{G}_{N}^{(\mathcal{B})}$ equipped with interaction matrix~\eqref{eqn:S_bots}, memory sets $\mathcal{M}^{(t)}_{i,j} = \{t-k_{ji}^{(1)},\cdots, t-k_{ji}^{(M)}\}$ for all $i,j \in \{1,\cdots,N\}$, and a bot set $\mathcal{B}= \{B_{1},\cdots,B_{K}\}$. For $\eta_{B_{k}}\neq p^{(*)}$ at least one $k \in \{1,\cdots,K\}$, the following holds:
\begin{align}\label{eqn:multi_bot_one}
(P^{(\mathcal{B},*)}-P^{(*)})= \mathbf{0}_{N} \iff \widetilde{S}_{\mathcal{B}}\boldsymbol{\eta}_{p^{(*)}} = \mathbf{0}_{N},
\end{align}
where $\boldsymbol{\eta}_{p^{(*)}} := [\eta_{B_{1}}-p^{(*)},\cdots,\eta_{B_{K}}-p^{(*)}]^{\mathbf{T}}$. 
\end{theorem}
\textit{Proof:}
The proof follows along the same lines as that of Theorem~\ref{thm:lower_bound_one_bot}. Corresponding to $\mathcal{G}_{N}^{(\mathcal{B})}$, we consider a homogeneous network $\mathcal{G}_{N}$ (without bots) with the same memory sets and the following two conditions on its interaction matrix $S=[s_{ij}]_{i,j =1}^{N}$:
\begin{itemize}
\item[(1)] $s_{ij} = 0 \iff \widetilde{s}_{ij}=0$,
\item[(2)] $\widetilde{s}_{ij}\leq s_{ij}$ for all $i,j \in \{1,\cdots,N\}$.
\end{itemize}
We now let $Y_{i}^{(\mathcal{B})}(t):= (P_{i}^{(\mathcal{B})}(t)-P_{i}(t))$, and use \eqref{eqn:P_calc} and \eqref{eqn:P_calc_bots} to write the follows:
\begin{align}\label{eqn:Y_calc_multi_bots}
P_{i}^{(\mathcal{B})}(t) - P_{i}(t) = c_{1,M}(\sum_{j=1}^{N}\widetilde{s}_{ij}-1) + \sum_{k=1}^{K}\widetilde{s}_{iB_{k}}\eta_{B_{k}} + c_{2,M} \sum_{j=1}^{N}\big[\sum_{n \in \mathcal{M}_{i,j}^{(t)}}\widetilde{s}_{ij}(P_{j}^{(\mathcal{B})}(t-n) - P_{j}(t-n))\big].
\end{align}
Similar to Theorem~\ref{thm:lower_bound_one_bot}, we denote :
\begin{align*}
Y_{i}^{(\mathcal{B},*)}:=\lim_{t \to \infty}(P_{i}^{(\mathcal{B})}(t)-P_{i}(t)); \quad Y^{(\mathcal{B},*)} := [Y_{1}^{(\mathcal{B},*)},\cdots,Y_{N}^{(\mathcal{B},*)}]^{\mathbf{T}}
\end{align*}
and take $t\to \infty$ to obtain the follows:
\begin{align}\label{eqn:multi_bots_proof}
Y^{(\mathcal{B},*)} = (I_{N}-c_{2,M}M\widetilde{S}_{N})^{-1}\widetilde{S}_{\mathcal{B}}\boldsymbol{\eta}_{p^{(*)}}.\end{align}
The proof follows from \eqref{eqn:multi_bots_proof}. Furthermore, note that \eqref{eqn:eta_equals_fixedpt} in Theorem~\ref{thm:lower_bound_one_bot} follows from \eqref{eqn:multi_bots_proof} by setting $\mathcal{B}= \{B\}$.
$\null\nobreak\hfill\ensuremath{\square}$

Note that, a non-trivial solution exists for $\widetilde{S}_{\mathcal{B}}\boldsymbol{\eta}_{p^{(*)}} = \mathbf{0}_{N}$ if and only if Rank of $\widetilde{S}_{\mathcal{B}}$ is strictly less than $K$ (number of bots in the network), which can only be achieved when there is linear dependency between row/columns, i.e., bot weights across individuals are proportional. Moreover the condition for cancellation of bot effects in multi-bot case obtained in Theorem~\ref{thm:multi_bots} requires the bias of at least two bots to favor opposing opinions. We formally state this in the Corollary below and further present a Remark concerning intuition behind cancellation of bot effects.

\begin{corollary}\label{cor:cancellation}
Under the assumptions of Theorem~\ref{thm:multi_bots}, suppose $\widetilde{S}_B \eta_{p^{(*)}} = \mathbf{0}_N$ with 
$\eta_{B_k} \neq p^{(*)}$ for at least one $k$. Then there exist indices $k_1, k_2 \in \{1,\ldots,K\}$ 
such that $\eta_{B_{k_1}} < p^{(*)}$ and $\eta_{B_{k_2}} > p^{(*)}$.
\end{corollary}

\textit{Proof.} We can 
write the cancellation condition in \eqref{eqn:multi_bot_one} as $\sum_{k=1}^K \widetilde{s}_{iB_k}(\eta_{B_k} - p^{(*)}) = 0$ for 
each $i \in \{1,\ldots,N\}$. Since $\widetilde{s}_{iB_k} \geq 0$ for all $i\in \{1,\cdots,N\}, k\in\{1,\cdots,K\}$, and at least one 
$\eta_{B_k} \neq p^{(*)}$, the sum $\sum_{k=1}^K \widetilde{s}_{iB_k}(\eta_{B_k} - p^{(*)})$ can equal zero only if the centered strengths $(\eta_{B_k} - p^{(*)})$ 
are not all of the same sign (for any row $i$ with at least two nonzero entries $\widetilde{s}_{iB_k}$). 
Hence at least one bot must have strength strictly above $p^{(*)}$ and at least one strictly below.
$\null\nobreak\hfill\ensuremath{\square}$

\begin{remark}\label{rem:bot_cancellation}
Theorem~\ref{thm:multi_bots} shows that exact cancellation occurs precisely when the centered strength vector belongs to the null space of the bot-influence matrix 
$\widetilde{S}_B$. Thus, a non-trivial cancellation is possible if and only if the bot-influence matrix $\widetilde{S}_B$ fails to have full rank. Equivalently, the rows/columns of $\widetilde{S}_B$ are linearly dependent, indicating that the bots do not possess linearly independent influence profiles across the network. In particular, when the number of bots equals the number of agents and $\widetilde{S}_B$ is invertible, the only solution is the trivial one with $\eta_{B_k}=p^{(*)}$ for every bot. More generally, equation~\eqref{eqn:multi_bots_proof} shows that the equilibrium perturbation is completely determined by $\widetilde{S}_B\eta_{p^{(*)}}$. Thus, exact cancellation is the limiting case in which this vector vanishes, while imperfect cancellation results in a corresponding perturbation of the equilibrium.
\end{remark}

\section{Simulation Results}\label{sec:simulations}
In this section, we provide a set of simulations to illustrate our model and results \footnote{ The numerical results presented in this section are reproducible using the code available in the accompanying GitHub repository {\url{https://github.com/Somya-Shiv-Nadar-University/opinion-dynamics}}}. In this regard, we consider homogeneous ($c_{1,M}$ and $c_{2,M}$ same for all agents, as given by \eqref{eqn:ratio_X}) as well as non-homogeneous (see \eqref{eqn:X_non_homo}) versions of our model. To illustrate the impact of bot on agents, we further present simulations for a network of agent equipped with a bot to demonstrate the model developed in section~\ref{sec:bots}. We consider a directed network of $10$ agents (see Fig. \ref{fig:network}) with edge weights given by a row-stochastic matrix. Unless stated otherwise, this network is equipped with finitely recent and time homogeneous memory sets $\mathcal{M}_{i,j}^{(t,10)}$ for all $i,j \in \{1,\cdots, 10\}$, such that $|\mathcal{M}_{i,j}^{(t)}|=5$ for all $t>10$ (i.e., $M=5$). Here the memory activation time (denoted by $T$ in section~\ref{sec:model}) is $10$ and we have set the memory depth $L=9$.   A plot for average ergodic sum versus time is plotted for each experiment, where   
\begin{align}\label{eqn:ergodic_sum}
I_{t}(i)=\frac{1}{t}\sum\limits_{n=1}^{t}Z_{i,n}
\end{align}
is the ergodic sum for agent $i$. The simulation results are obtained by averaging over $5000$ independent realizations of the process defined by \eqref{eqn:ratio_X}. Although the plots generated from $5000$ independent runs retain some noise, they are sufficient to accurately illustrate the long-term behavior of the process. Consequently, a larger number of runs was not considered. Furthermore, we compute the ergodic sum $I_{t}(i)$ as defined in \eqref{eqn:ergodic_sum} up to time $200$ for each agent $i$ for all the experiments. As stated in remark~\ref{rem:Markov}, the underlying Markovian process is ergodic, and therefore $\eqref{eqn:ergodic_sum}$ tends to the fixed point of \eqref{eqn:P_eqn} for the homogeneous case and $i$th entry of \eqref{eqn:fixed_pt_vector_non_homo} for the non-homogeneous case. This fixed point is the asymptotic belief of an agent.

\begin{figure}
\begin{center}
\includegraphics[scale=0.5]{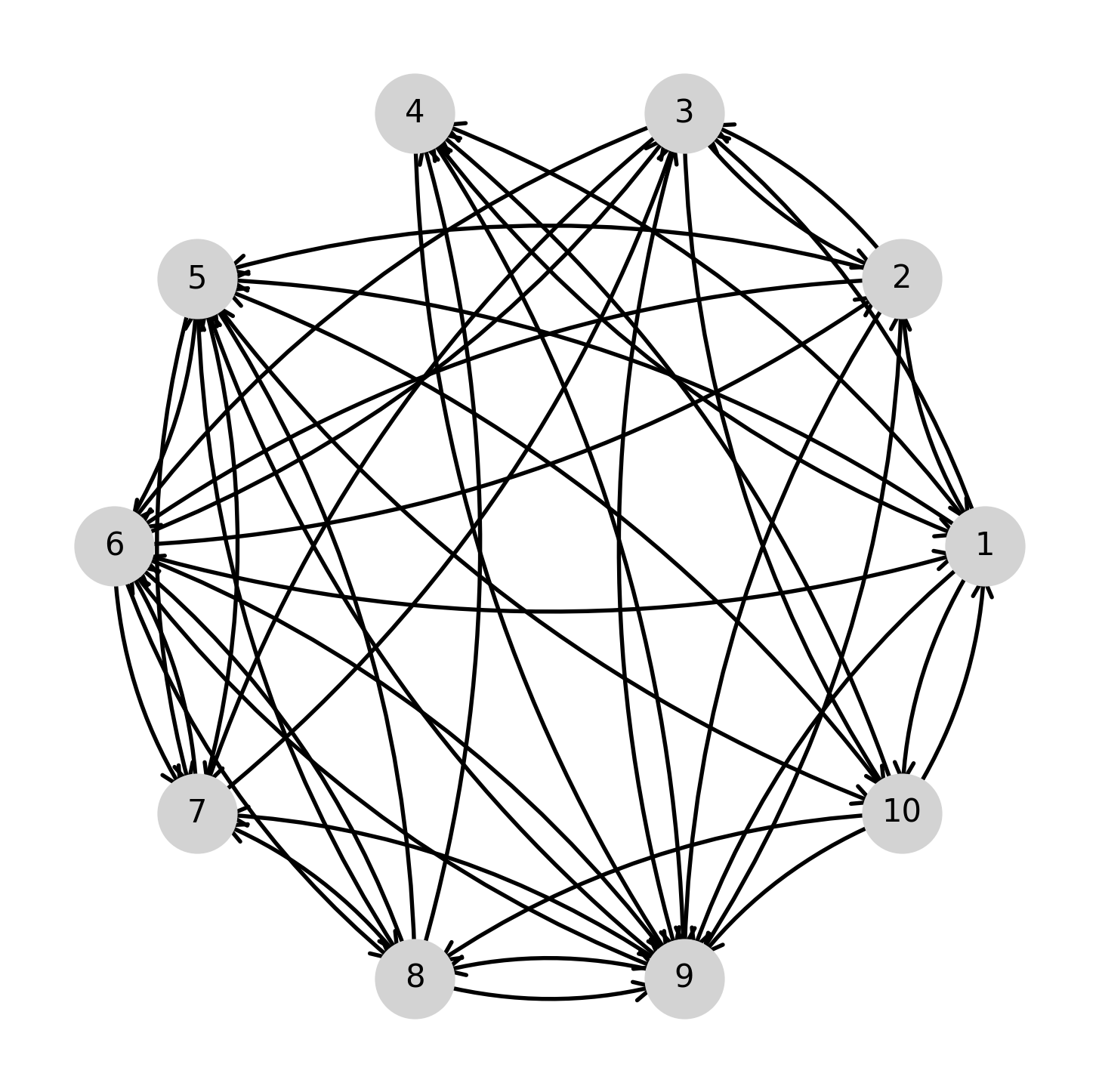}   
\caption{A $10$-node network of agents. The weights on the directed edges are given by a row stochastic matrix as defined in section~\ref{sec:model}. We use this network to generate all our simulation plots.}
\label{fig:network}
\end{center}
\end{figure}

\begin{figure}
 \centering \includegraphics[scale=0.6]{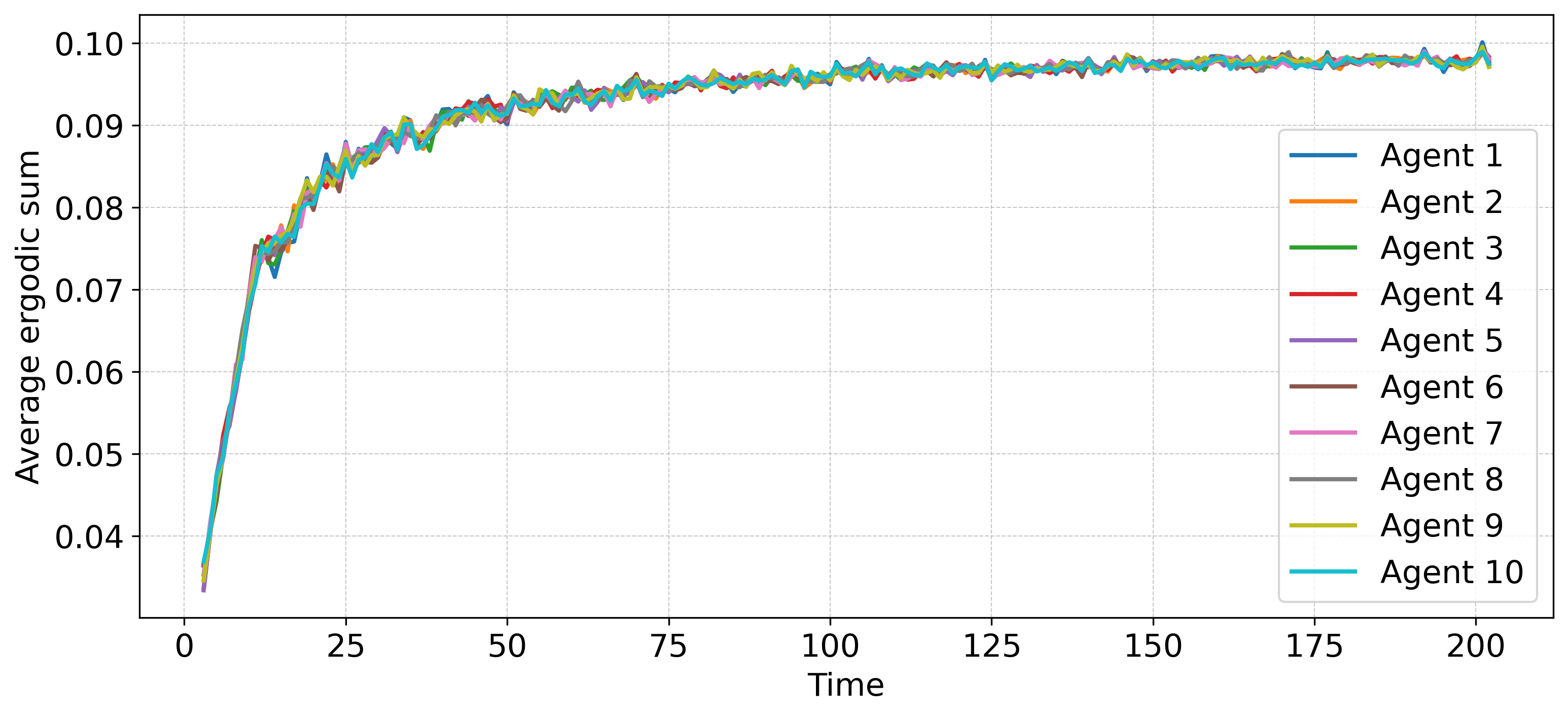}
 \caption{Average ergodic sum versus time plot for a homogeneous network with $10$ agents and $M=5$. The model parameters $c_{1,M} = 0.05$ and $c_{2,M}=0.1$ in \eqref{eqn:ratio_X}. The belief of each agent tends to $0.1$, which is consistent with the fixed point $p^{(*)}$ obtained in \eqref{eqn:homo_fixed_point}}  
 \label{fig:homo}
\end{figure}

In Fig.~\ref{fig:homo}, we plot the average ergodic sum versus time for a homogeneous $10$ - node network. The models parameters $c_{1,M}=0.05$ and $c_{2,M}= 0.1$ in \eqref{eqn:ratio_X} and $M=5$. As concluded in section~\ref{sec:Analysis}, the fixed point for each agent is the same irrespective of the network  structure and edge weights. The observed fixed point for this simulation is $0.1$ which is consistent with the theoretical fixed point obtained in \eqref{eqn:ratio_X}. In Fig.~\ref{fig:non_homo}, we simulate average ergodic sum for non-homogeneous $10$ - node network (i.e., node dependent parameters $c^{(i)}_{1,M}$ and $c^{(i)}_{2,M}$ for $i \in \{1,\cdots,10\}$ in \eqref{eqn:X_non_homo}). As discussed in Remark~\ref{rem:non_homo}, the fixed point for each node in the non-homogeneous case depends on its interactions with other agents. In order to observe the impact of memory on agents, we further simulate the average ergodic sum for this network in Fig.~\ref{fig:memory_comparison}, for $M=2,3$ keeping all the other parameters same. For visual clarity, only the average ergodic sums corresponding to the first three agents are displayed. A crucial observation is that the value of the fixed point for each agent decreases when the memory factor $M$ is decreased. Moreover, the effect of increasing $M$ is qualitatively similar across all three agents. Therefore, although a closed-form expression for the fixed-point vector is not available in the non-homogeneous case, Fig.~\ref{fig:memory_comparison} shows that for each of the three displayed agents, the asymptotic belief 
decreases as $M$ decreases from $5$ to $3$ to $2$, with the ordering preserved across agents. 
This is consistent with the closed-form formula in the homogeneous case~\eqref{eqn:homo_fixed_point}, where 
$p^{(*)} = c_{1,M}/(1 - c_{2,M}M)$ is strictly increasing in $M$ for $c_{2,M} > 0$.

\begin{figure}
\centering
\includegraphics[scale=0.6]{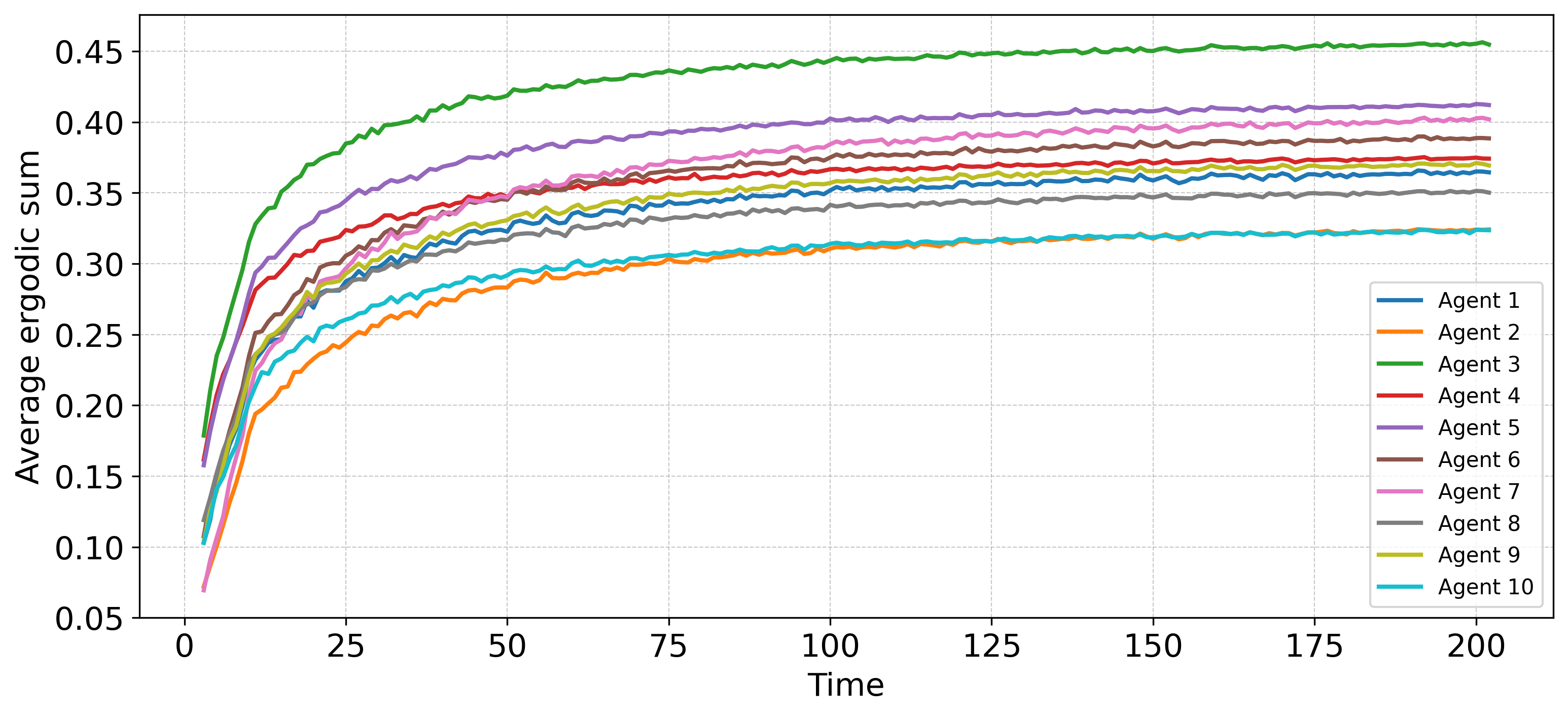}
\caption{Average ergodic sum versus time plot for a non-homogeneous network with $10$ agents with $M=5$. Unlike the homogeneous case, each agent attains a different asymptotic belief, which depends on the interaction parameters.}
\label{fig:non_homo}
\end{figure}

\begin{figure}
\centering\includegraphics[scale=0.6]{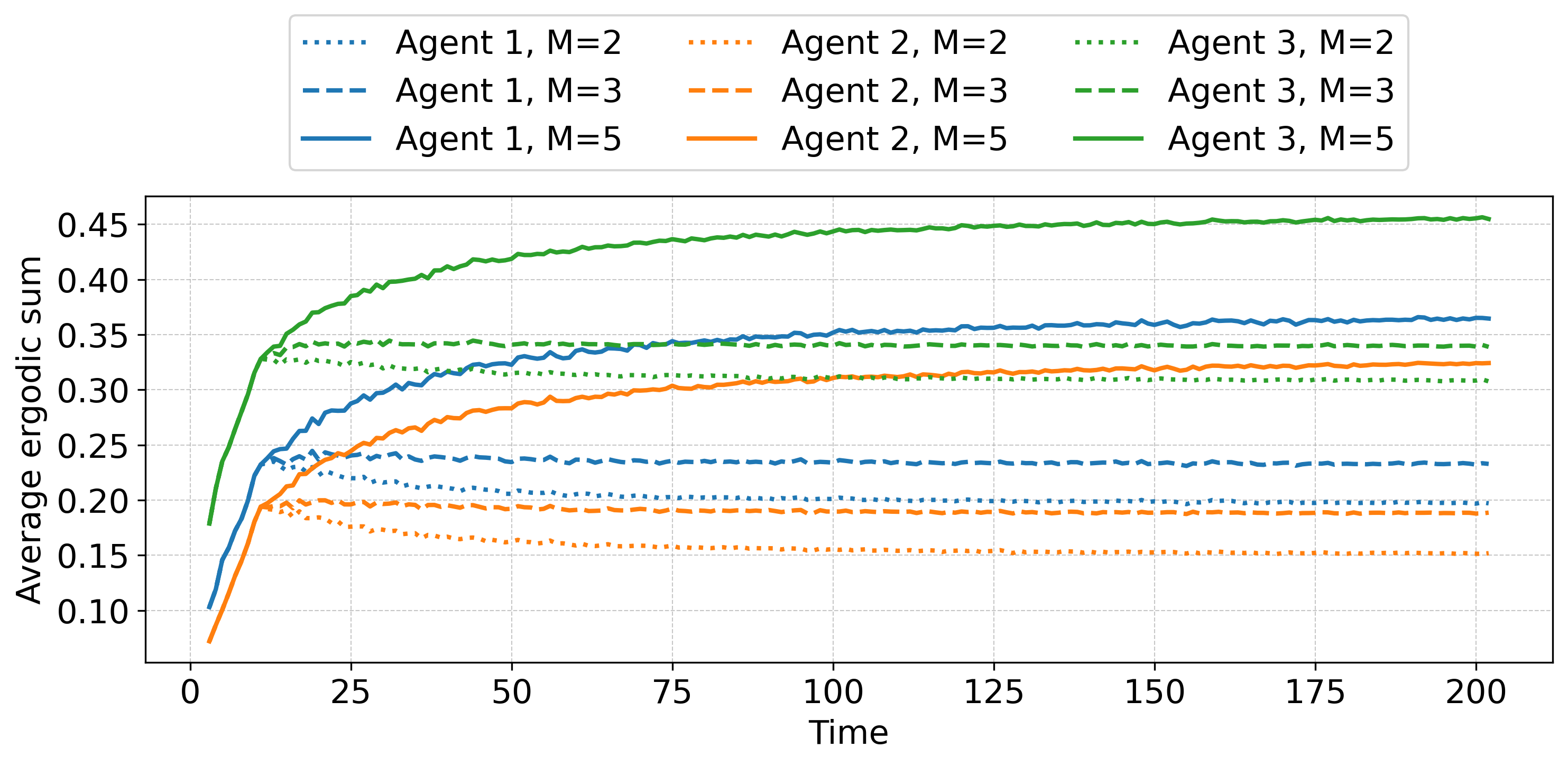}
\caption{In this figure, we plot the average ergodic sum for the first three agents of the non-homogeneous network in Fig.~\ref{fig:network} with different memories. We observe that the value of the fixed point decreases with decreasing memory.}
\label{fig:memory_comparison}
\end{figure}

\begin{table}[h!]
\centering
\resizebox{0.6\textwidth}{!}{\begin{tabular}{|c|c|c|}
\hline
\textbf{Agent} & \textbf{Average Ergodic sum at $\mathbf{t=200}$ in Fig.~\ref{fig:non_homo}} & \textbf{Average Ergodic sum at $\mathbf{t=200}$ in Fig.~\ref{fig:nonhomo_bot}} \\
\hline
1 & 0.36 & 0.42\\
\hline
2 & 0.32& 0.39\\
\hline
3 & 0.45 & 0.49\\
\hline
4 & 0.37 & 0.42\\
\hline
5 & 0.41 & 0.45\\
\hline
6 & 0.39 & 0.44\\
\hline
7 & 0.40 & 0.46\\
\hline
8 & 0.35 & 0.40\\
\hline
9 & 0.37 & 0.42\\
\hline
10 & 0.32 & 0.38\\
\hline
\end{tabular}}

\vspace{0.1cm}

\caption{A comparison of fixed points simulated in Fig.~\ref{fig:non_homo} and Fig.~\ref{fig:nonhomo_bot} for a $10$ - node non-homogeneous network in absence and presence of bot respectively. Since $\eta_{B} =0.5$, which is greater than the fixed points of all the agents in absence of bots, the bot shifts the limiting bias of each agent towards $1$. We observe that this increase in limiting bias is different for different agents, and depends on factors such as initial parameters, interaction weight of the bot and influence of other agents. }
\label{tab:nonhomo_compare}
\end{table}

In the next set of simulations, we add a bot (denoted by $B$) to Fig.~\ref{fig:network} such that it uniformly affects all the agents. In particular,
\begin{align}\label{eqn:S_bots_simulation}
\widetilde{S}^{(\mathcal{B})}_{N} :=\left[ 
\begin{array}{c|c} 
  \begin{array}{c} (1-\beta)S_{N}\end{array} & \beta\mathbf{1}_{N}\\ 
  \hline 
  \mathbf{0}_{1\times N} & 1 
\end{array} \right],
\end{align}
 where $S_{N}$ is the interaction matrix for Fig~\ref{fig:network} and $\beta \in (0,1)$. The average ergodic sum in the presence of a bot is given by:
 \begin{align}\label{eqn:ergodic_sum_bot}
I_{t}^{(B)}(i)=\frac{1}{t}\sum\limits_{n=1}^{t}Z^{(B)}_{i,n},
 \end{align}
where $Z^{(B)}_{i,n}$ are generated using \eqref{eqn:drawing_bots}. In Fig.~\ref{fig:homo_bot_two} and Fig.~\ref{fig:homo_bot}, we generate the ergodic sum \eqref{eqn:ergodic_sum_bot} (averaged over $5000$ runs) for homogeneous Fig.~\ref{fig:network} with bot strengths $\eta_{B}=0.05$ and $\eta_{B}=0.8$ respectively ($c_{1,M}$ and $c_{2,M}$ are same as the experiment in Fig.~\ref{fig:homo}). We set the bot weight $\beta = 0.2$ in both the experiments. Both these figures indicate that a uniform bot effect across all agents in a homogeneous network does not disrupt the consensus. However, as Theorem~\ref{thm:lower_bound_one_bot} predicts, this consensus value obtained is now lower than $p^{(*)}$ in Fig.~\ref{fig:homo_bot_two} because $\eta_{B}< p^{(*)}$. Similarly, since $\eta_{B}>p^{(*)}$ in Fig.~\ref{fig:homo_bot}, the fixed point for all the agents is same but lower than $p^{(*)}$. We also give simulation results for the non-homogeneous network in Fig.~\ref{fig:non_homo} equipped with an added bot with $\beta = 0.2$ and $\eta_{B} = 0.5$. A comparison of agent limiting biases for Fig.~\ref{fig:non_homo} and Fig.~\ref{fig:nonhomo_bot} is presented in TABLE~\ref{tab:nonhomo_compare}. We observe that even though the bot weight is uniform across the network, the bot effect varies across the agents. For instance, the bot has increased the asymptotic bias of agent $7$ more than it has increased it for agent $5$. The former has a higher asymptotic bias than the latter in the presence of bot in Fig.~\ref{fig:nonhomo_bot}, while the trend is opposite in the absence of bot in Fig.~\ref{fig:non_homo}. Since, these experiments are run on a general non-homogeneous network, it is analytically challenging to completely quantify these shifts in agents beliefs when a bot is added. However, as demonstrated via examples in section~\ref{sec:bots}, these shifts in the fixed points can be computed for certain symmetric networks.  
 \begin{figure}
 \centering
\includegraphics[scale=0.6]{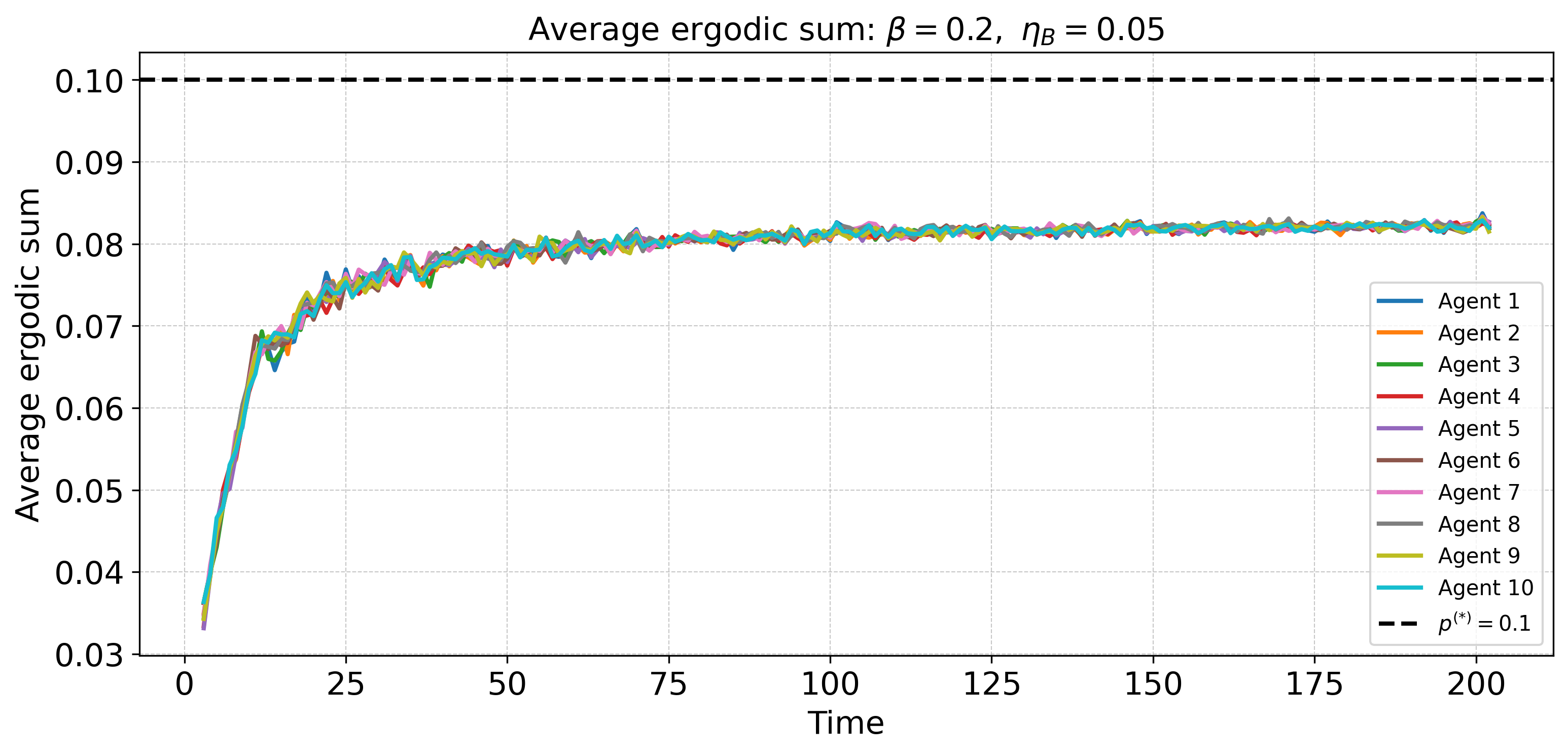}
\caption{Average ergodic sum versus time for the homogeneous network of Fig.~\ref{fig:homo}, augmented with a bot of strength $\eta_B = 0.05$ and uniform bot weight $\beta = 0.2$. The model parameters are $c_{1,M} = 0.05$, $c_{2,M} = 0.1$, $M = 5$, giving $p^{(*)} = c_{1,M}/(1-c_{2,M}M) = 0.1$. Since $\eta_B = 0.05 < p^{(*)} = 0.1$, Theorem~\ref{thm:lower_bound_one_bot} predicts $(P^{(B,*)} - P^{(*)}) \prec \mathbf{0}_N$: the bot shifts consensus downward. The observed consensus value $\approx 0.08$ is consistent with this prediction. Furthermore, the addition of bot does not disrupt the consensus achieved by the agents in a homogeneous network.}
\label{fig:homo_bot_two}
 \end{figure}

 \begin{figure}
 \centering
 \includegraphics[scale=0.6]{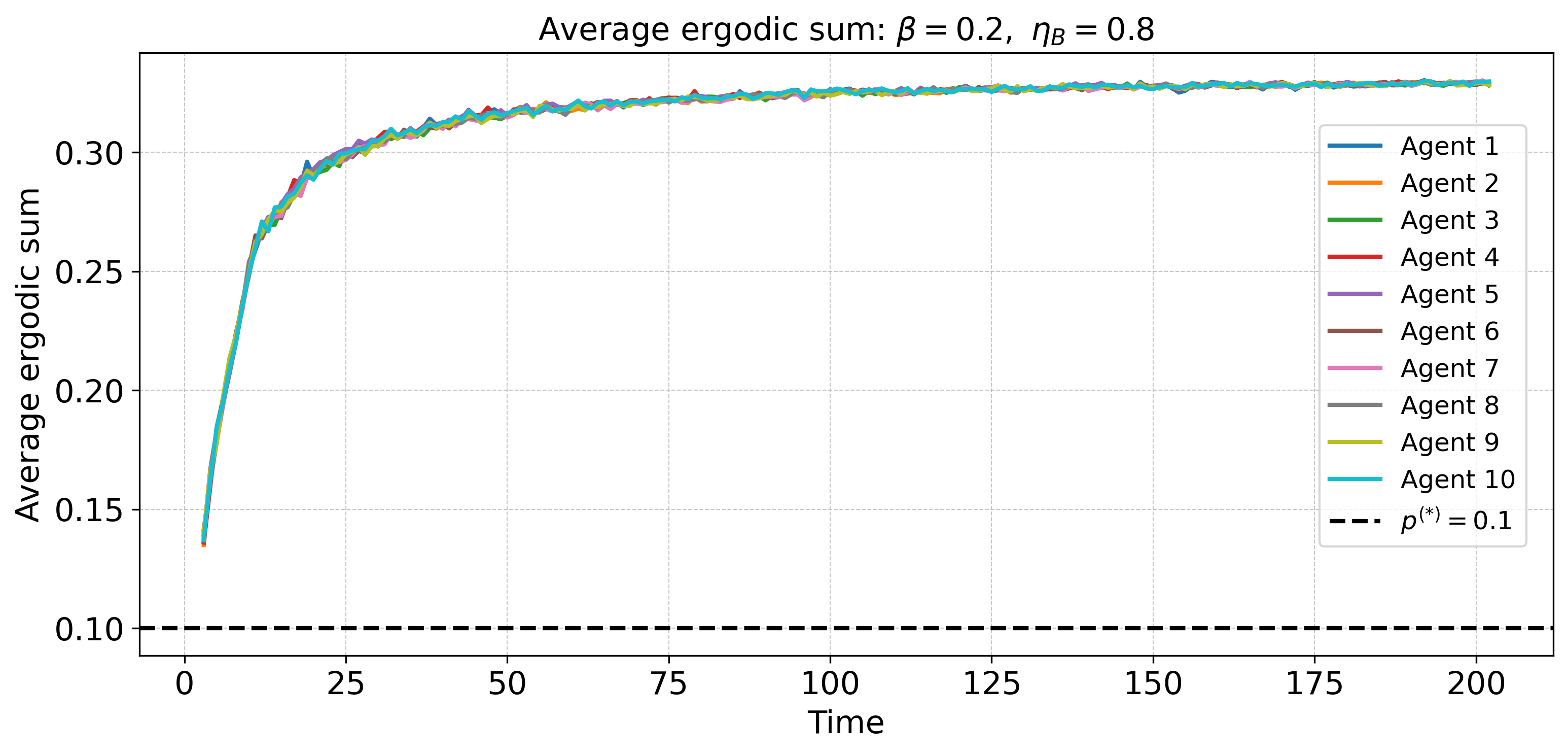}
 \caption{Average ergodic sum versus time for the same homogeneous network and bot-weight configuration as Fig.~\ref{fig:homo_bot_two} ($\beta = 0.2$, $c_{1,M} = 0.05$, $c_{2,M} = 0.1$, $M=5$, $p^{(*)} = 0.1$), but with bot strength $\eta_B = 0.8 > p^{(*)}$. By Theorem~\ref{thm:lower_bound_one_bot}, $(P^{(B,*)} - P^{(*)}) \succ \mathbf{0}_N$: the bot shifts consensus upward. The observed consensus value $\approx 0.33$ confirms this. In both experiments, the bot does not disrupt consensus among agents; it only shifts the common consensus value.}
 \label{fig:homo_bot}
 \end{figure}

\begin{figure}
\centering
\includegraphics[scale=0.6]{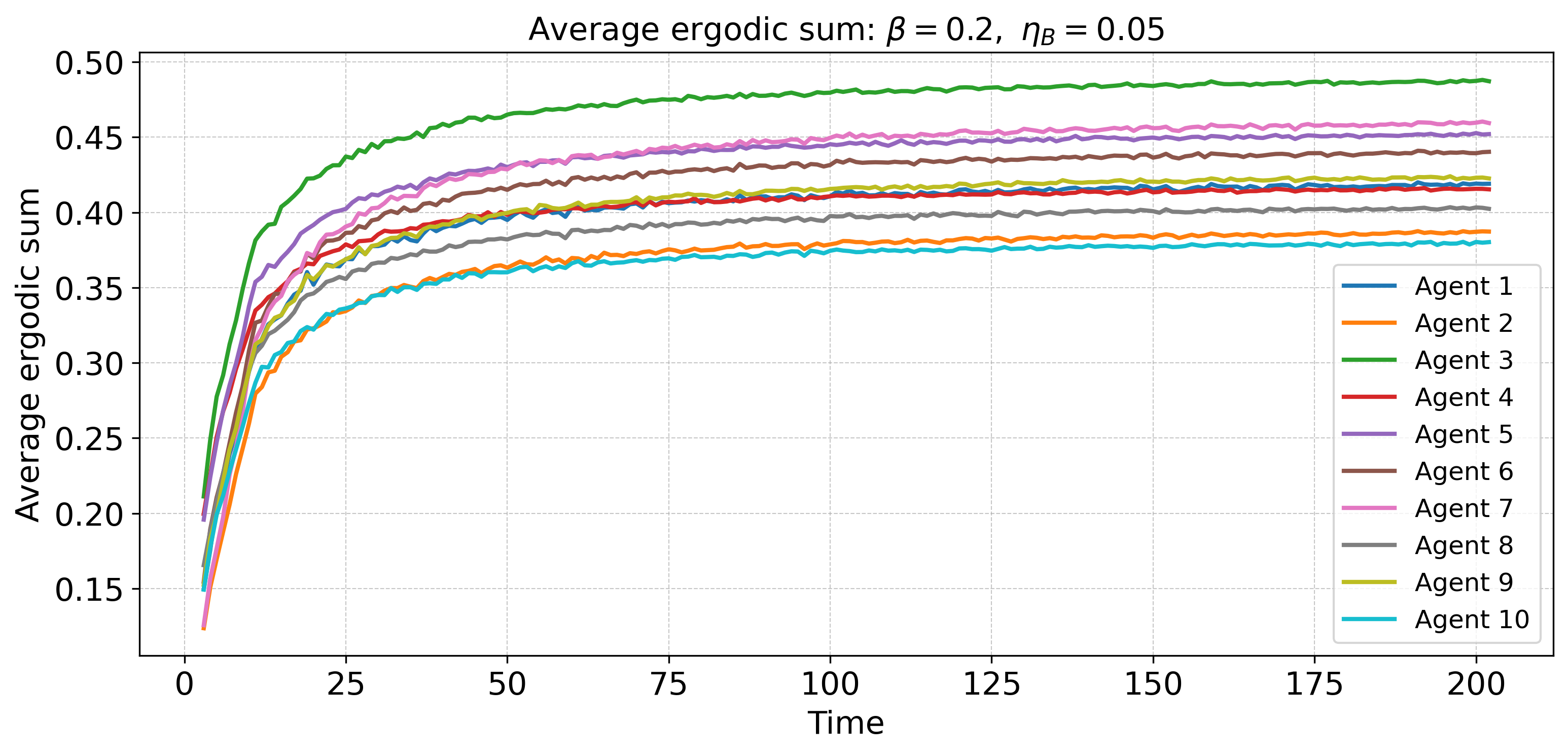}
\caption{In this figure, we present the plot of average ergodic sum versus time for the non-homogeneous network in Fig.~\ref{fig:non_homo} with an added bot with $\beta = 0.2$ and bot strength $\eta_{B} = 0.5$. The bot effect on each node varies and is analytically difficult to obtain in terms of model parameters and interactions.}
\label{fig:nonhomo_bot}
\end{figure}
 
To assess robustness of the simulation findings to network topology, we repeat the experiment of Fig.~\ref{fig:homo_bot_two} (homogeneous network, bot with $\eta_B < p^{(*)}$) 
on the Hub-and-Spoke network analyzed in Examples~\ref{eg:HS_bot_tohub} and~\ref{eg:HS_bot_toall}, with $N = 10$ nodes (one hub, nine spokes) 
and the interaction matrix given by~\eqref{eqn:S_one_symmetric_bot}. We set $c_{1,M} = 0.05$, $c_{2,M} = 0.1$, $M = 5$, and 
bot weight $\beta = 0.2$, matching the parameters of Fig~\ref{fig:homo_bot_two}. The resulting average ergodic sums 
are displayed in Figure~\ref{fig:homo_bot_hub_spoke}. As predicted by the theoretical analysis in Section~\ref{sec:bots} (Example~\ref{eg:HS_bot_tohub}), the hub and spokes converge 
to distinct fixed points $\beta_{H,M}$ and $\beta_{S,M}$ given by~\eqref{eqn:fixed_pt_hub_bot_tohub_M} and~\eqref{eqn:fixed_pt_spoke_bot_tohub} respectively, 
in contrast to the consensus observed in the general network of Fig.~\ref{fig:homo_bot_two}. The qualitative observation 
that a bot with $\eta_B < p^{(*)}$ shifts the hub belief downward (below its bot-free value) is consistent 
across both topologies, supporting the generality of Theorem~\ref{thm:lower_bound_one_bot}. However, the magnitude of the shift 
differs: the hub-and-spoke structure concentrates the bot's influence through the hub, resulting in 
a smaller perturbation to the spokes compared to the uniform-influence setting of Fig~\ref{fig:homo_bot_two}. 
These observations confirm that while the threshold $\eta_B > p^{(*)}$ in Theorem~\ref{thm:lower_bound_one_bot} is topology-independent, 
the \emph{magnitude} of the fixed-point shift depends on the network structure.
 
\begin{figure}
\centering
\includegraphics[scale=0.6]{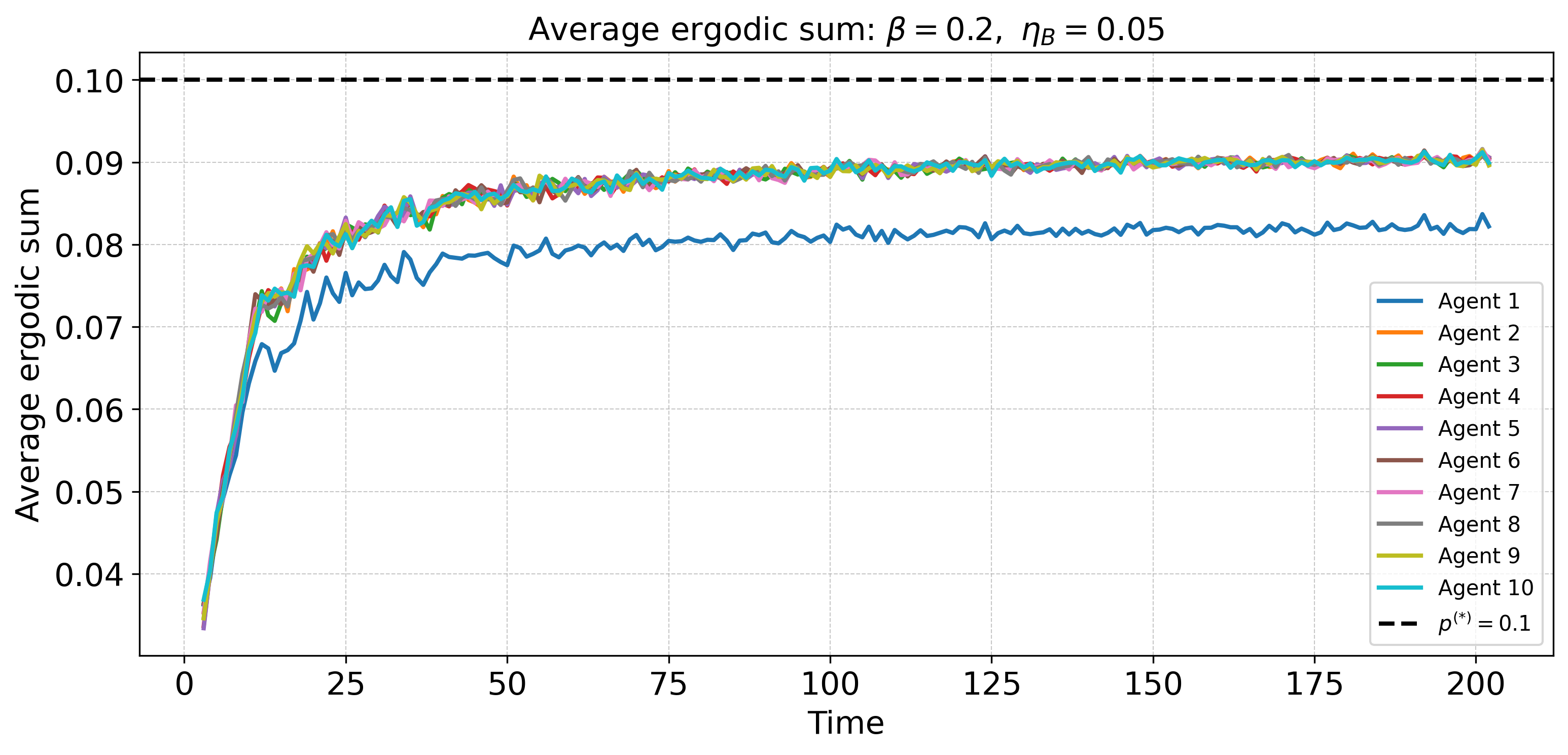}
\caption{Average ergodic sum versus time for a $10$-node hub-and-spoke network equipped with a bot connected only to the hub (as depicted in Fig.~\ref{fig:hub_spoke_bot_tohub}).}
\label{fig:homo_bot_hub_spoke}
\end{figure}

\section{Conclusions and Future Work}\label{sec:conclusions}
In this paper, we proposed an opinion dynamics model over a social network of interacting agents. The novelty of the model lies in the structure provided by memory sets over the agents. We constructed a suitable class of time-delayed dynamical system representing evolution of agent beliefs in the model and studied its stability properties. For the homogeneous case, i.e., when the parameters are same for all the agents, we obtain the fixed point (which is the asymptotic belief of the agents) for this dynamical system. We extended this setup to include bots in the network, and analyzed their effect on the asymptotic belief of the agents. We further presented a detailed discussion of Hub-and-Spoke model equipped with bots. For a homogeneous network equipped with one bot, we established a lower bound on the bot strength required for the bot to shift the agent beliefs towards its bias. In the multiple-bot setting, we obtain conditions on bot weights and bot strengths for various bots to nullify each other. Finally, simulation results are presented to demonstrate the asymptotic behavior of our model, influence of memory and bot effects. Future work includes extending the proposed memory sets to a time-inhomogeneous framework, constructing suitable martingales for the expressed-opinion process to derive fluctuation results and quantifying bot effects on agent beliefs in non-homogeneous networks. A further open direction is the rigorous characterization of fixed-point monotonicity in 
memory~$M$ for non-homogeneous networks. Simulation evidence (Fig.~\ref{fig:memory_comparison}) and analytical 
verification for $N = 2, M = 1$ support Conjecture~\ref{thm:conjecture}, but a general proof likely requires 
new techniques for analyzing the Schur complement structure of $(I_{NL \times NL} - J_N^{(L)})^{-1}$ 
as $L$ and the indicator matrices $\mathbf{M}_N^{(k)}$ vary with~$M$. Several other interesting questions remain open. A quantitative spectral analysis of the delayed system matrix could provide explicit convergence rates and reveal how different memory architectures influence asymptotic behavior. Another natural direction is to study the design of memory structures that optimize consensus speed or robustness against external influence. It would also be interesting to extend the present framework to adaptive or time-varying memory sets, dynamically evolving interaction networks, and strategic or adversarial bots whose behavior changes over time.

\bibliographystyle{IEEEtran}
\bibliography{references_opinion}

\end{spacing}
\end{document}